\documentclass[12pt,twoside]{amsart}
\usepackage{latexsym,amsfonts,amsmath,amssymb,soul}
\usepackage{enumerate,soul}
\usepackage[titletoc]{appendix}
\usepackage[usenames, dvipsnames]{color}
\numberwithin{equation}{section}

\newtheorem{Th}{Theorem}[section]
\newtheorem{Rem}[Th]{Remark}
\newtheorem{Ex}[Th]{Example}
\newtheorem{Lemma}[Th]{Lemma}
\newtheorem{Def}[Th]{Definition}
\newtheorem{Prop}[Th]{Proposition}
\newtheorem{Cor}[Th]{Corollary}

\catcode`\@=11

\renewcommand{\section}%
   {\setcounter{equation}{0}\@startsection {section}{1}{\z@}{-3.5ex plus -1ex
  minus -.2ex}{2.3ex plus .2ex}{\Large\bf}}

\def\supp{\mathop{\rm supp}\nolimits}
\def\Re{\mathop{\rm Re}\nolimits}
\def\Im{\mathop{\rm Im}\nolimits}
\def\Wig{\mathop{\rm Wig}\nolimits}

\def\PW{\mathop{\rm PW}\nolimits}
\def\PWG{\mathop{\rm PWG}\nolimits}

\def\union{\mathop{\cup}}
\def\ds{\displaystyle}
\def\R{\mathbb R}
\def\C{\mathbb C}
\def\N{\mathbb N}

\newcommand{\D}{\mathcal{D}}
\newcommand{\E}{\mathcal{E}}
\newcommand{\F}{\mathcal{F}}

\newcommand{\Sch}{\mathcal{S}}

\newcommand{\afrac}[2]{\genfrac{}{}{0pt}{1}{#1}{#2}}

\newcommand{\beqsn}{\arraycolsep1.5pt\begin{eqnarray*}}
\newcommand{\eeqsn}{\end{eqnarray*}\arraycolsep5pt}
\newcommand{\beqs}{\arraycolsep1.5pt\begin{eqnarray}}
\newcommand{\eeqs}{\end{eqnarray}\arraycolsep5pt}

\title{Real Paley-Wiener theorems in spaces of ultradifferentiable functions}

\author[Boiti]{Chiara Boiti}
\address{
Dipartimento di Matematica e Informatica \\Universit\`a di Ferrara\\
Via Ma\-chia\-vel\-li n.~30\\
I-44121 Ferrara\\
Italy}
\email{chiara.boiti@unife.it}

\author[Jornet]{David Jornet}
\address{
Instituto Universitario de Matem\'atica Pura y Aplicada IUMPA\\
Universitat Po\-li\-t\`ecni\-ca de Val\`encia\\
Camino de Vera, s/n\\
E-46071 Valencia\\
Spain}
\email{djornet@mat.upv.es}

\author[Oliaro]{Alessandro Oliaro}
\address{Dipartimento di Matematica\\ Universit\`a di Torino\\
 Via Carlo Alberto n. 10\\ I-10123 Torino\\ Italy}
 \email{alessandro.oliaro@unito.it}

\begin{document}


\begin{abstract}
We develop real Paley-Wiener theorems for classes $\Sch_\omega$ of ultradifferentiable functions and related $L^{p}$-spaces in the spirit of Bang and Andersen for the Schwartz class. We introduce results of this type for the so-called Gabor transform and give a full characterization in terms of Fourier and Wigner transforms for several variables of a Paley-Wiener theorem in this general setting, which is new in the literature. We also analyze this type of results when the support of the function is not compact using polynomials. Some examples are given.

\end{abstract}

\maketitle

 \newcount\minutes
 \newcount\hour
 \newcount\minutea
 \newcount\minuteb
 \minutes=\time
  \divide\minutes by 60
  \hour=\minutes
\minutes=\time
 \multiply \hour by 60
  \advance \minutes by -\hour
 \divide \hour by 60
 \minuteb=\minutes
 \divide\minuteb by 10
\minutea=\minuteb
\multiply \minuteb by 10
\advance \minutes by  -\minuteb
\minuteb=\minutes

\markboth{\sc Real Paley-Wiener theorems in spaces of ultradifferentiable
 functions }
 {\sc C.~Boiti, D.~Jornet and A.~Oliaro}

\section{Introduction}

As stated in \cite{AD}, ``A Paley-Wiener theorem is a characterization, by relating support to growth, of the image of a space of functions or distributions under a transform of Fourier type.'' This relation comes only in terms of a compact and convex set in which the support of the function or distribution is included. In fact, the growth of $\hat{f}$ on $\C^{d}$ enables to retrieve the convex hull of the support of $f$, but no more precise information can be obtained from it (see \cite{AD} and the references therein). In the last years, a new type of results called ``real Paley-Wiener type theorems'' has received much attention, which try to circumvent this theoretical obstruction for the classical Paley-Wiener theorems to ``look inside'' the convex hull of the support. The word ``real'' expresses that information about the support of $f$ comes from growth rates associated to the function $\hat{f}$ on $\R^{d}$ rather than on $\C^{d}$ as in the classical ``complex Paley-Wiener theorems''. This theory was initiated by Bang and Tuan, and here we follow the approach of Andersen and Andersen-De Jeu (see \cite{A1,A2,AD,B,Tu} and the references therein), who state results of ``real Paley-Wiener'' type in spaces of rapidly decreasing functions (the Schwartz class $\Sch(\R^{d})$) or in $L^{p}$ spaces in their most general version, using polynomials, where the support of the function (or distribution) could be non-compact or even non-convex.

Bj\"orck~\cite{Bj} introduced in 1966 {\em global} classes of ultradifferentiable functions $\Sch_{\omega}(\R^{d})$ using weights $\omega$ in the sense of Beurling to extend previous theorems of H\"ormander about interior regularity of linear partial differential operators with constant coefficients. These weight functions permit to treat in a unified way a big scale of classes of functions or (ultra)distributions and are especially suitable for manipulations on the Fourier transform side. We recall here that when the weight function is  the logarithm, i.e. $\omega(t)=\log(1+t)$, the class $\Sch_{\omega}$ is the Schwartz class $\Sch$.   In the last 60 years, the classes of ultradifferentiable functions and their duals have been intensively studied for very different purposes and have become the right setting to study many different problems in analysis in a very general way (partial differential equations, Paley-Wiener theorems, Whitney jets, Borel theorems, etc.). We mention \cite{BMT} as the reference for the modern point of view of the treatment of these classes where the authors get, under some conditions on the weight functions, to relate the growth of the functions in terms of their partial derivatives and the growth of their Fourier transforms, property that has many advantages.

As Andersen and De Jeu mention in \cite{AD}, their theorems of ``real'' type can be extended to other transforms of Fourier type, where the classical theorems cannot. In fact, also to more general spaces of functions as we will show below. Our aim is to study real Paley-Wiener theorems in the spirit of Bang, Andersen and Andersen and De Jeu~\cite{A1,A2,AD,B} in the more general $\Sch_{\omega}$-setting and related $L^{p}$-spaces. Moreover, we show that some transforms coming from the field of time-frequency analysis enter into the game, like the Gabor and Wigner transforms. We also study the case when the support of the Fourier transform is not necessarily compact or convex, extending some results in terms of polynomials in the spirit of \cite{Tu,AD}.

In Section 2 we give some preliminaries and definitions on weight functions, Fourier type transforms and the space $\Sch_{\omega}(\R^{d})$ especially when the seminorms are given in terms of $L^{p}$-norms. In Section 3 we extend \cite[Theorem 1]{B} for several variables in the $\Sch_{\omega}$-setting in different ways (see Proposition~\ref{th3A1}). Also in this section we state a general version of \cite[Theorem 1]{A1} for the ultradifferentiable setting and several variables (Theorem~\ref{th1A1}). Our main result in this section is Theorem~\ref{th4A2}, where we give a full characterization of the known ``complex Paley-Wiener theorem'' in the Beurling setting (see \cite[Proposition 3.4(2)]{BMT})  in terms of Wigner transforms; in this result, we assume that the support of the Fourier transform of the $\Sch_{\omega}$-function is inside a hypercube in $\R^{d}$. To obtain it, we need some preparation: to study the behaviour of the Gabor transform  of a function $f$ in $\Sch_{\omega}$ with respect to a window $\psi\in\Sch_{\omega}$, in a suitable weighted mixed $L^{p,q}$-space, in terms of the support of the function $f$ and the window $\psi$ (Proposition~\ref{prop1}). As a consequence, the symmetric properties of the Wigner transform give surprising results (Corollaries \ref{corWig1} and \ref{corWig2}). We finish this section with an example about Hermite functions. In Section 4 we treat the case of arbitrary support and, following the lines of \cite{AD}, we extend Theorem 2.2 and 2.5 of this paper (these are our Theorem~\ref{th22AD} and Corollary~\ref{cor-AND}). Finally, in Example~\ref{rem2} we analyze the relation of the definition of the generalized support \eqref{5} with the regularity of the corresponding polynomial.

\section{Preliminaries}
We begin with the definition of non-quasianalytic weight function in the sense of \cite{BMT} suitable for the Beurling case, i.e. we consider the logarithm as a weight function also.
\begin{Def}
\label{defomega}
A {\em non-quasianalytic weight function} is a continuous increasing function
$\omega:\ [0,+\infty)\to[0,+\infty)$ satisfying the following properties:
\begin{itemize}
\item[$(\alpha)$]
There exists $L\ge 1$ such that $\omega(2t)\leq L(\omega(t)+1),\quad \forall t\geq 0$;
\vspace*{1mm}
\item[$(\beta)$]
$\ds \int_1^{+\infty}\frac{\omega(t)}{t^2}dt<+\infty$;
\vspace*{1mm}
\item[$(\gamma)$]
there exist $a\in\R$ and $b>0$ such that
\beqsn
\omega(t)\geq a+b\log(1+t),\qquad\forall t\geq0.
\eeqsn
\item[$(\delta)$]
$\varphi(t):=\omega(e^t)$ is convex.
\end{itemize}
Then, for $\zeta\in\C^d$, we define $\omega(\zeta):=\omega(|\zeta|)$.
\end{Def}

\begin{Rem}\label{add10}
We recall some well-known properties on weight functions; the proofs can be found in the literature, we recall them here for the sake of completeness.
\begin{itemize}
\item[(i)]
Condition $(\alpha)$ implies that for every $t_1,t_2\geq 0$
\beqs
\label{add1}
\omega(t_1+t_2)\leq L(\omega(t_1)+\omega(t_2)+1);
\eeqs
indeed, since $\omega$ is increasing and positive we have
\beqsn
\omega(t_1+t_2)\leq \omega(2\max\{t_1,t_2\})\leq L( \omega(\max\{t_1,t_2\})+1)\leq L(\omega(t_1)+\omega(t_2)+1).
\eeqsn
\item[(ii)]
Since \eqref{add1} trivially implies $(\alpha)$ with $2L$ instead of $L$, we have that $(\alpha)$ is equivalent to \eqref{add1} (cf. \cite{BMT}).
\item[(iii)]
By condition $(\alpha)$ and \eqref{add1} we easily deduce that for every $k\in\mathbb{N}$ and $t\geq 0$,
\beqs
\label{add3}
\omega(k t)\leq D_k (\omega(t)+1),
\eeqs
where $D_k=L+L^2+\dots+L^{k-1}$.
\item[(iv)]
By $(\beta)$ and the fact that $\omega$ is increasing, we have that $\omega(t)=o(t)$ as $t\to +\infty$ (cf. \cite{MT}). This can be deduced by the fact that
\beqsn
\frac{\omega(t)}{t}=\int_t^{+\infty} \frac{\omega(t)}{s^2}\,ds\leq\int_t^{+\infty} \frac{\omega(s)}{s^2}\,ds.
\eeqsn
\item[(v)]
By condition $(\gamma)$ we have
\beqs
\label{add2}
e^{-\sigma\omega(t)}\in L^p(\R^d),\quad \forall \sigma\geq \frac{d+1}{bp}.
\eeqs
\end{itemize}
\end{Rem}

We denote by $\varphi^*$ the {\em Young conjugate} of $\varphi$, defined by
\beqs
\label{young}
\varphi^*(s):=\sup_{t\geq0}\{ts-\varphi(t)\}.
\eeqs
We recall that it is an increasing convex function
satisfying $\varphi^{**}=\varphi$ (see \cite{hor-cvx}). We will use throughout the next Lemma (easy to prove; see \cite{BMT}).
\begin{Lemma}
 \label{phistar}
Let $\omega:\ [0,+\infty)\to[0,+\infty)$ be a continuous increasing function
such that $\varphi(t):=\omega(e^t)$ is convex. Then the following
properties hold:
\begin{enumerate}[{(i)}]
\item
$\varphi^*(s)/s$ is increasing.
\item
$\varphi^*(t)+\varphi^*(s)\leq\varphi^*(t+s),\quad t,s\geq0$.
\item
If there exist $A\geq0$ and $B\geq1$ such that
$\omega(et)\leq A+B\omega(t)$ for all $t\geq0$, then for all
$\lambda>0$ and $j,n\in\N_0=\N\cup\{0\}$:
\beqsn
\lambda\varphi^*\left(\frac j\lambda\right)+nj\leq
\frac{\lambda}{B^n}\varphi^*\left(\frac{j}{\lambda/B^n}\right)+\lambda n
\frac AB.
\eeqsn
Note that if $\omega$ is subadditive (that means it satisfies $\omega(t_1+t_2)\leq \omega(t_1)+\omega(t_2)$ for every $t_1,t_2\geq 0$), then we can take $A=0$ and $B=3$.
\item
If there exist $A\geq0$ and $B\geq1$ such that
$\omega(et)\leq A+B\omega(t)$ for all $t\geq0$, then for all
$\rho,\lambda>0$ and $j\in\N_0$:
\beqsn
\rho^j e^{\lambda\varphi^*\left(\frac j\lambda\right)}\leq
\Lambda_{\rho,\lambda}e^{\lambda'\varphi^*\left(\frac{j}{\lambda'}\right)}
\eeqsn
for all $0<\lambda'\leq\lambda/B^{[\log\rho+1]}$ and
$\Lambda_{\rho,\lambda}=e^{\lambda\frac AB[\log\rho+1]}$, where
$[x]$ denotes the integer part of $x$.
\item
For all $\lambda>0$ and $k\in\N_0$:
\beqsn
&&t^ke^{-\lambda\varphi^*\left(\frac k\lambda\right)}\leq
e^{\lambda\omega(t)},\qquad t\geq1,\\
&&t^ke^{-\lambda\varphi^*\left(\frac k\lambda\right)}\leq e^{-\lambda\varphi^*(0)}
e^{\lambda\omega(t)},\qquad 0\leq t\leq 1.
\eeqsn
\item
If there exist $a\in\R$ and $b>0$ such that
$\omega(t)\geq a+b\log(1+t)$ for all $t\geq0$, then for all $\sigma,\lambda>0$
and $t\geq1$:
\beqsn
\inf_{j\in\N_0}t^{-\sigma j}e^{\lambda\varphi^*\left(\frac{\sigma j}{\lambda}
\right)}\leq e^{-\left(\lambda-\frac\sigma b\right)\omega(t)-a\frac\sigma b}.
\eeqsn
\item
If $\omega(t)=o(t)$ as $t$ tends to infinity, for every $\ell\in\N$ there exists a constant $C_\ell>0$ such that
\beqsn
s\log s\leq s+\ell\varphi^*\left(\frac s\ell\right)+C_\ell,\qquad s>0.
\eeqsn
\item
Assume that there exist $A\geq0$ and $B\geq1$ such that
$\omega(et)\leq A+B\omega(t)$ for all $t\geq0$, and moreover  $\omega(t)=o(t)$ as $t$ tends to infinity.
Then, for all
$D,\lambda>0$ and $n\in\N_0$:
\beqsn
D^n n!\leq C_{D,\lambda}e^{\lambda\varphi^*\left(\frac n\lambda\right)},
\eeqsn
for some $C_{D,\lambda}>0$.
\item
For all $j,h,r\in\N_0$ and $\lambda>0$:
\beqsn
e^{\lambda\varphi^*\left(\frac{j}{\lambda}\right)}
e^{\lambda\varphi^*\left(\frac{r+h}{\lambda}\right)}\leq
e^{\frac\lambda2\varphi^*\left(\frac{j+h}{\lambda/2}\right)}
e^{\frac\lambda2\varphi^*\left(\frac{r}{\lambda/2}\right)}.
\eeqsn
\end{enumerate}
\end{Lemma}

In this paper we will consider classes of {\em ultradifferentiable functions} of Beurling type in the sense of Braun, Meise and Taylor~\cite{BMT}, which are defined, for a weight function $\omega$ and an open subset $\Omega$ of $\R^{d}$, by
\beqsn
\E_{(\omega)}(\Omega):=\Big\{f\in C^\infty(\R^d):&&\mbox{for each } K\subset\!\subset\Omega \mbox{ and for each }\ \lambda>0,\\
&&\ \ \ \ \sup_{\alpha\in\N_0^d}\sup_{x\in K}|D^\alpha f(x)|
e^{-\lambda\varphi^*\left(\frac{|\alpha|}{\lambda}\right)}<+\infty\Big\},
\eeqsn
where $D^\alpha=D_1^{\alpha_1}
\cdots D_d^{\alpha_d}$ with $D_j=-i\partial_{x_j}$. Here, we relax condition $(\gamma)$ of \cite[Definition 1.1]{BMT} in our Definition~\ref{defomega} since we consider only Beurling classes (as Bj\"orck~\cite{Bj}, but considering more general weights that are not necessarily subadditive).

Then, the space of {\em ultradifferentiable functions} of Beurling type
{\em with compact support} in $\Omega$ is denoted by $\D_{(\omega)}(\Omega),$ and its corresponding dual space by $\D'_{(\omega)}(\Omega)$, which is called the space of {\em ultradistributions} of Beurling type.

We consider also  the {\em Fourier transform} of $u
\in L^1(\R^d)$ denoted by
\beqsn
\F(u)(\xi)=\hat{u}(\xi):=\int_{\R^d}u(x)e^{-i\langle x,\xi\rangle}dx,\qquad\xi\in\R^d,
\eeqsn
with standard extensions to more general spaces of functions and distributions. The so-called {\em short-time Fourier transform}
(or {\em Gabor transform}) of $u\in {L^2}(\R^d)$, for a window function
$\psi\in {L^2}(\R^d)$, is denoted by
\beqsn
V_\psi u(z):=\int_{\R^d}u(y)\overline{\psi(y-x)}e^{-i\langle y,\xi\rangle}dy,
\qquad z=(x,\xi)\in\R^{2d}.
\eeqsn
The {\em Wigner transform} of $u,v\in L^2(\R^d)$ is denoted by
\beqsn
\Wig (u,v)(x,\xi):=\int_{\R^d}u\left(x+\frac t2\right)
\overline{v\left(x-\frac t2\right)}
e^{-i\langle\xi,t\rangle}dt, \qquad x,\xi\in\R^d.
\eeqsn
Then we write $\Wig u$ for $\Wig(u,u)$.
We refer to \cite{G} for the classical properties of the
Gabor and Wigner transforms. The setting of this work is given by the following definition.

\begin{Def}[\cite{Bj}]
\label{defSomega}
The space $\Sch_\omega(\R^d)$ is the set of all $u\in L^1(\R^d)$ such that $u,\hat{u}\in C^\infty(\R^d)$
and for each $\lambda>0$ and each $\alpha\in\N_0^d$ we have
$$\sup_{x\in\R^d}e^{\lambda\omega(x)}
|D^\alpha u(x)|<+\infty\quad\mbox{and}\quad \sup_{\xi\in\R^d}e^{\lambda\omega(\xi)}
|D^\alpha \hat{u}(\xi)|<+\infty.$$
The corresponding strong dual of ultradistributions will be denoted by $\Sch'_\omega(\R^d)$.
\end{Def}

By condition $(\gamma)$ of Definition~\ref{defomega} it is easy to deduce  that $\Sch_\omega(\R^d)\subset\Sch(\R^d)$. Hence, $\Sch_\omega(\R^d)$ can be equivalently defined as the set of all $u\in\Sch(\R^d)$ that satisfy the condition of Definition \ref{defSomega}. By Bj\"orck~\cite{Bj}, we know that the Fourier transform $\mathcal{F}:\Sch_\omega(\R^d)\to \Sch_\omega(\R^d)$ is a continuous automorphism, that can be extended in the usual way to $\Sch_\omega^\prime(\R^d)$ and, moreover, the space $\Sch_\omega(\R^d)$ is an algebra under multiplication and convolution. On the other hand, for $u,\psi\in\Sch_\omega(\R^d)$ we have $V_\psi u, \Wig u\in \Sch_\omega(\R^{2d})$. Moreover, for $u,\psi\in\Sch_\omega^\prime(\R^d)$ the Gabor and Wigner transforms are well defined and belong to $\Sch_\omega^\prime(\R^{2d})$ \cite{GZ, BJO1,BJO2}. We recall, for the reader convenience, the following result \cite{BJO1,BJO2}.
 \begin{Th}
 \label{thSomega}
 {Given $u\in\Sch(\R^d)$,  $u\in\Sch_\omega(\R^d)$
 if and only if one
of the following conditions hold:}
\begin{itemize}
\item[(a)]
\begin{itemize}
\item[i)]
$\forall\lambda>0,\alpha\in\N_0^d\ \exists C_{\alpha,\lambda}>0$ s.t.
$\ds\sup_{x\in\R^d}e^{\lambda\omega(x)}|D^\alpha u(x)|\leq C_{\alpha,\lambda}\,$,
\item[ii)]
$\forall\lambda>0,\alpha\in\N_0^d\ \exists C_{\alpha,\lambda}>0$ s.t.
$\ds\sup_{\xi\in\R^d}e^{\lambda\omega(\xi)}|D^\alpha \hat{u}(\xi)|\leq C_{\alpha,\lambda}\,$;
\end{itemize}
\item[(b)]
\begin{itemize}
\item[i)]
$\forall\lambda>0,\alpha\in\N_0^d\ \exists C_{\alpha,\lambda}>0$ s.t.
$\ds\sup_{x\in\R^d}e^{\lambda\omega(x)}|x^\alpha u(x)|\leq C_{\alpha,\lambda}\,$,
\item[ii)]
$\forall\lambda>0,\alpha\in\N_0^d\ \exists C_{\alpha,\lambda}>0$ s.t.
$\ds\sup_{\xi\in\R^d}e^{\lambda\omega(\xi)}|\xi^\alpha \hat{u}(\xi)|\leq C_{\alpha,\lambda}\,$;
\end{itemize}
\item[(c)]
\begin{itemize}
\item[i)]
$\forall\lambda>0\ \exists C_{\lambda}>0$ s.t.
$\ds\sup_{x\in\R^d}e^{\lambda\omega(x)}| u(x)|\leq C_{\lambda}\,$,
\item[ii)]
$\forall\lambda>0\ \exists C_{\lambda}>0$ s.t.
$\ds\sup_{\xi\in\R^d}e^{\lambda\omega(\xi)}|\hat{u}(\xi)|\leq C_{\lambda}\,$;
\end{itemize}
\item[(d)]
\begin{itemize}
\item[i)]
$\forall\lambda>0,\beta\in\N_0^d\ \exists C_{\beta,\lambda}>0$ s.t.
$\ds\sup_{\alpha\in\N_0^d}\sup_{x\in\R^d}|x^\beta D^\alpha u(x)|e^{-\lambda
\varphi^*\left(\frac{|\alpha|}{\lambda}\right)}\leq C_{\beta,\lambda}\,$,
\item[ii)]
$\forall\mu>0,\alpha\in\N_0^d\ \exists C_{\alpha,\mu}>0$ s.t.
$\ds\sup_{\beta\in\N_0^d}\sup_{x\in\R^d}|x^\beta D^\alpha u(x)|e^{-\mu
\varphi^*\left(\frac{|\beta|}{\mu}\right)}\leq C_{\alpha,\mu}\,$;
\end{itemize}
\item[(e)]
$\forall\mu,\lambda>0\ \exists C_{\mu,\lambda}>0$ s.t.
$\ds\sup_{\alpha,\beta\in\N_0^d}\sup_{x\in\R^d}|x^\beta D^\alpha u(x)|
e^{-\lambda\varphi^*\left(\frac{|\alpha|}{\lambda}\right)}e^{-\mu
\varphi^*\left(\frac{|\beta|}{\mu}\right)}\leq C_{\mu,\lambda}\,$;
\item[(f)]
$\forall\lambda>0\ \exists C_{\lambda}>0$ s.t.
$\ds\sup_{\alpha,\beta\in\N_0^d}\sup_{x\in\R^d}|x^\beta D^\alpha u(x)|e^{-\lambda
\varphi^*\left(\frac{|\alpha+\beta|}{\lambda}\right)}\leq C_{\lambda}\,$;
\item[(g)]
$\forall\mu,\lambda>0\ \exists C_{\mu,\lambda}>0$ s.t.
$\ds\sup_{\alpha\in\N_0^d}\sup_{x\in\R^d}|D^\alpha u(x)|
e^{-\lambda\varphi^*\left(\frac{|\alpha|}{\lambda}\right)}e^{\mu\omega(x)}
\leq C_{\mu,\lambda}\,$;
\item[(h)]
{given $\psi\in\Sch_\omega(\R^d)\setminus\{0\}$, }
$\forall\lambda>0\ \exists C_{\lambda}>0$ s.t.
$\ds\sup_{z\in\R^{2d}}|V_\psi u(z)|e^{\lambda\omega(z)}\leq C_\lambda$.
\end{itemize}
\end{Th}


In the following, it is sometimes more convenient to use $L^p$-norms instead of $L^\infty$-norms
in $\Sch_\omega(\R^d)$. We need the following notation of $L^{p,q}$-space:
\beqs
\label{40}
L^{p,q}(\R^{2d}):=\Big\{&&\!\!F \mbox{ measurable on $\R^{2d}$ such that:}\\
\notag
&&\|F\|_{L^{p,q}}:=\Big(\int_{\R^d}\Big(\int_{\R^d}|F(x,\xi)|^pdx\Big)^{q/p}d\xi\Big)^{1/q}
<+\infty\Big\}
\eeqs
if $1\leq p,q<+\infty$; here, we replace the $L^p$ or $L^q$ norm with the essential supremum if $p$ or $q$ is equal to $\infty$. We obtain the next extension of Theorem~\ref{thSomega}:
\begin{Th}
\label{propSomega}
Given a function $u\in\Sch(\R^d)$ and $1\leq p,q\leq+\infty$, we have that $u\in\Sch_\omega(\R^d)$ if and only
if one of the following conditions is satisfied:
\begin{itemize}
\item[$(a)'$]
\begin{itemize}
\item[i)]
$\forall\lambda>0,\ \alpha\in\N^d_0\ \exists C_{\alpha,\lambda}>0$ s.t.
$\ds\|e^{\lambda\omega(x)} D^\alpha u(x)\|_{L^p}\leq C_{\alpha,\lambda}\,$,
\item[ii)]
$\forall\lambda>0,\ \alpha\in\N^d_0\ \exists C_{\alpha,\lambda}>0$ s.t.
$\ds\|e^{\lambda\omega(\xi)} D^\alpha\hat{u}(\xi)\|_{L^q}\leq C_{\alpha,\lambda}\,$;
\end{itemize}
\item[$(b)'$]
\begin{itemize}
\item[i)]
$\forall\lambda>0,\ \alpha\in\N^d_0\ \exists C_{\alpha,\lambda}>0$ s.t.
$\ds\|e^{\lambda\omega(x)} x^\alpha u(x)\|_{L^p}\leq C_{\alpha,\lambda}\,$,
\item[ii)]
$\forall\lambda>0,\ \alpha\in\N^d_0\ \exists C_{\alpha,\lambda}>0$ s.t.
$\ds\|e^{\lambda\omega(\xi)} \xi^\alpha\hat{u}(\xi)\|_{L^q}\leq C_{\alpha,\lambda}\,$;
\end{itemize}
\item[$(c)'$]
\begin{itemize}
\item[i)]
$\forall\lambda>0\ \exists C_{\lambda}>0$ s.t.
$\ds\|e^{\lambda\omega(x)} u(x)\|_{L^p}\leq C_{\lambda}\,$,
\item[ii)]
$\forall\lambda>0\ \exists C_{\lambda}>0$ s.t.
$\ds\|e^{\lambda\omega(\xi)}\hat{u}(\xi)\|_{L^q}\leq C_{\lambda}\,$;
\end{itemize}
\item[$(d)'$]
\begin{itemize}
\item[i)]
$\forall\lambda>0,\beta\in\N_0^d\ \exists C_{\beta,\lambda}>0$ s.t.
$\ds\sup_{\alpha\in\N_0^d}\|x^\beta D^\alpha u(x)\|_{L^p}e^{-\lambda
\varphi^*\left(\frac{|\alpha|}{\lambda}\right)}\leq C_{\beta,\lambda}\,$,
\item[ii)]
$\forall\mu>0,\alpha\in\N_0^d\ \exists C_{\alpha,\mu}>0$ s.t.
$\ds\sup_{\beta\in\N_0^d}\|x^\beta D^\alpha u(x)\|_{L^q}e^{-\mu
\varphi^*\left(\frac{|\beta|}{\mu}\right)}\leq C_{\alpha,\mu}\,$;
\end{itemize}
\item[$(e)'$]
$\forall\mu,\lambda>0\ \exists C_{\mu,\lambda}>0$ s.t.
$\ds\sup_{\alpha,\beta\in\N_0^d}\|x^\beta D^\alpha u(x)\|_{L^p}
e^{-\lambda\varphi^*\left(\frac{|\alpha|}{\lambda}\right)}e^{-\mu
\varphi^*\left(\frac{|\beta|}{\mu}\right)}\leq C_{\mu,\lambda}\,$;
\item[$(f)'$]
$\forall\lambda>0\ \exists C_{\lambda}>0$ s.t.
$\ds\sup_{\alpha,\beta\in\N_0^d}\|x^\beta D^\alpha u(x)\|_{L^p}e^{-\lambda
\varphi^*\left(\frac{|\alpha+\beta|}{\lambda}\right)}\leq C_{\lambda}\,$;
\item[$(g)'$]
$\forall\mu,\lambda>0\ \exists C_{\mu,\lambda}>0$ s.t.
$\ds\sup_{\alpha\in\N_0^d}\|e^{\mu\omega(x)} D^\alpha u(x)\|_{L^p}
e^{-\lambda\varphi^*\left(\frac{|\alpha|}{\lambda}\right)}\leq C_{\mu,\lambda}\,$;
\item[$(h)'$]
Given $\psi\in\Sch_\omega(\R^d)\setminus\{0\}$, $\forall\lambda>0\ \exists C_{\lambda}>0$ s.t.
$\ds\|V_\psi u(z) e^{\lambda\omega(z)}\|_{L^{p,q}}\leq C_\lambda$.
\end{itemize}
\end{Th}

\begin{proof}
\underline{$(c)'\Leftrightarrow u\in\Sch_\omega(\R^d)$}:\\
Let us assume $u\in\Sch(\R^d)$ to satisfy $(c)'$ and prove that
$u\in\Sch_\omega(\R^d)$.
To this aim we shall prove that $u$ satisfies condition $(h)$ of
Theorem~\ref{thSomega}, for some
fixed $\psi\in\Sch_\omega(\R^d)\setminus\{0\}$.
We fix $\sigma\geq(d+1)/bp'$, where $b$ is the constant in condition $(\gamma)$
of Definition \ref{defomega} and $p'$ is the conjugate exponent of $p$. Let us first compute
\beqs
\nonumber
|e^{\lambda\omega(x)}V_\psi u(x,\xi)|=&&\left|
e^{\lambda\omega(x)}\int_{\R^d}u(y)\overline{\psi(y-x)}e^{-i\langle y,\xi\rangle}dy\right|\\
\nonumber
\leq&&e^{\lambda L}\int_{\R^d}|u(y)|e^{\lambda L\omega(y)}
|\psi(y-x)|e^{\lambda L\omega(y-x)+\sigma\omega(y-x)}e^{-\sigma\omega(y-x)}dy\\
\nonumber
\leq&&e^{\lambda L}\|e^{(\lambda L+\sigma)\omega(t)}\psi(t)\|_{L^\infty}\cdot\|e^{\lambda L\omega(y)}u(y)\|_{L^p}
\cdot\|e^{-\sigma\omega(t)}\|_{L^{p'}}\\
\label{304''}
\leq&&C_\lambda
\eeqs
since $\psi\in\Sch_\omega(\R^d)$ and because of $(c)'(i)$, \eqref{add1} and \eqref{add2}.

On the other hand (see \cite[formula (3.10)]{G}),
\beqsn
V_\psi u(x,\xi)=e^{-i\langle x,\xi\rangle}V_{\hat\psi}\hat{u}(\xi,-x),
\eeqsn
so that, as in \eqref{304''} with $q$ instead of $p$, we obtain that also
\beqs
\label{304'}
|e^{\lambda\omega(\xi)}V_\psi u(x,\xi)|=
|e^{\lambda\omega(\xi)}V_{\hat\psi} \hat{u}(\xi,-x)|\leq C'_\lambda
\eeqs
for some $C'_\lambda>0$.

Then, from \eqref{304''}, \eqref{304'} and \eqref{add1}:
\beqsn
|V_\psi u(x,\xi)|=&&\sqrt{|V_\psi u(x,\xi)|^2}
\leq\sqrt{C_\lambda e^{-\lambda\omega(x)}C'_\lambda e^{-\lambda\omega(\xi)}}\\
\leq&& C''_\lambda e^{-\frac\lambda2(\omega(x)+\omega(\xi))}
\leq C''_\lambda e^{\frac{\lambda}{2}} e^{-\frac{\lambda}{2L}\omega(x,\xi)}
\eeqsn
for some $C''_\lambda>0$, i.e. condition $(h)$ of Theorem~\ref{thSomega} is
satisfied and $u\in\Sch_\omega(\R^d)$.

Conversely, if $u\in\Sch_\omega(\R^d)$ then condition $(c)$ of Theorem~\ref{thSomega}
is satisfied and hence from \eqref{add2}
\beqsn
\|e^{\lambda\omega(x)}u(x)\|_{L^p}\leq\|e^{(\lambda+\sigma)\omega(x)}u(x)\|_{L^\infty}
\cdot\|e^{-\sigma\omega(x)}\|_{L^p}\leq C_\lambda
\eeqsn
for $\sigma\geq (d+1)/bp$, and analogously $\|e^{\lambda\omega(\xi)}\hat{u}(\xi)\|_{L^q}\leq C_\lambda$ for some
$C_\lambda>0$. \\[0.2cm]
\indent\underline{$(a)'\Leftrightarrow u\in\Sch_\omega(\R^d)$}:\\
If $u$ satisfies $(a)'$, then it satisfies $(c)'$, so from the previous point $u\in \Sch_\omega(\R^d)$. On the other hand, if $u\in \Sch_\omega(\R^d)$, from $(a)$ of Theorem \ref{thSomega} we have
\beqsn
\Vert e^{\lambda\omega(x)} D^\alpha u(x)\Vert_{L^p}\leq \Vert e^{(\lambda+\sigma)\omega(x)} D^\alpha u(x)\Vert_{L^\infty} \Vert e^{-\sigma\omega(x)}\Vert_{L^p}\leq C^\prime_{\alpha,\lambda},
\eeqsn
for $\sigma\geq (d+1)/bp$, so $(a)'(i)$ is satisfied; the proof of $(a)'(ii)$ is similar. \\[0.2cm]
\indent\underline{$(b)'\Leftrightarrow u\in\Sch_\omega(\R^d)$}:\\
It is enough to prove that $(b)'\Leftrightarrow (c)'$. Since $(b)'\Rightarrow (c)'$ is trivial, let us suppose that $u$ satisfies $(c)'$; from the condition $(\gamma)$ of Definition \ref{defomega}, for $c=1/b$ and $C_\alpha=e^{-a|\alpha|/b}$, we have
\beqsn
\vert e^{\lambda\omega(x)} x^\alpha\vert\leq e^{\lambda\omega(x)+\vert\alpha\vert\log\vert x\vert}\leq C_\alpha e^{(\lambda+c\vert\alpha\vert)\omega(x)}.
\eeqsn
Hence, we obtain
\beqsn
\Vert e^{\lambda\omega(x)} x^\alpha u(x)\Vert_{L^p}\leq C_\alpha \Vert e^{(\lambda+c\vert\alpha\vert)\omega(x)} u(x)\Vert_{L^p}\leq C_{\alpha,\lambda}
\eeqsn
for some $C_{\alpha,\lambda}>0$,
so that $(b)'(i)$ is satisfied. Analogously we get $(b)'(ii)$. \\[0.2cm]
\indent\underline{$(f)'\Leftrightarrow u\in\Sch_\omega(\R^d)$}:\\
Let $u\in\Sch(\R^d)$ which satisfies $(f)'$. It is enough to see that $\hat{u}\in\Sch_\omega(\R^d)$. For all $\xi\in\R^d$, $\alpha,\beta\in\N_0^d$:
\beqsn
|\xi^\beta D^\alpha\hat{u}(\xi)|=&&\left|\F\big(D_x^\beta(x^\alpha u(x))\big)(\xi)\right|
\leq\|D^\beta x^\alpha u\|_{L^1}\\
\leq&&\|(1+|x|^2)^{-n}\|_{L^{p'}}\cdot\|(1+|x|^2)^nD^\beta(x^\alpha u(x))\|_{L^p}\\
\leq&&C_n\sum_{\afrac{\gamma\leq\beta}{\gamma\leq\alpha}}
\binom{\beta}{\gamma}\frac{\alpha!}{(\alpha-\gamma)!}
\|(1+|x|^2)^nx^{\alpha-\gamma}D^{\beta-\gamma}u(x)\|_{L^p}
\eeqsn
for some $C_n>0$ if we choose $n>d/(2p')$. Therefore, by $(f)'$, it is easy to see (Lemma~\ref{phistar}) that
for every $\lambda>0$ there exists $C_{\lambda}>0$ such that for each $\xi\in\R^{d}$,
\beqsn
|\xi^\beta D^\alpha\hat{u}(\xi)|\leq C_{\lambda}
e^{\lambda\varphi^*\left(\frac{|\alpha+\beta|}{\lambda}\right)}.
\eeqsn

In the other direction, if $u\in\Sch_\omega(\R^d)$, we have, by Lemma \ref{phistar},
\beqsn
\|x^\beta D^\alpha u(x)\|_{L^p}\leq&& \|(1+|x|^2)^{-n}\|_{L^p}\cdot
\|(1+|x|^2)^nx^\beta D^\alpha u(x)\|_{L^\infty}\\
\leq&& C_nC^\prime_{2\lambda}e^{2\lambda\varphi^*\left(\frac{|\alpha+\beta|+2n}{2\lambda}\right)}
\leq C_n C''_{2\lambda}e^{\lambda\varphi^*\left(\frac{|\alpha+\beta|}{\lambda}\right)}
e^{\lambda\varphi^*\left(\frac{2n}{\lambda}\right)}
\leq \widetilde{C}_\lambda e^{\lambda\varphi^*\left(\frac{|\alpha+\beta|}{\lambda}\right)}
\eeqsn
for some $C_n,C^\prime_{2\lambda},C''_{2\lambda},\widetilde{C}_\lambda>0$ if we choose $n>d/(2p)$. \\[0.2cm]
\indent\underline{$(e)'\Leftrightarrow u\in\Sch_\omega(\R^d)$}: \\
From the convexity of $\varphi^*$ we get that $(e)'\Leftrightarrow(f)'$, cf. Lemma~\ref{phistar}(ii) and (ix). \\[0.2cm]
\indent\underline{$(g)'\Leftrightarrow u\in\Sch_\omega(\R^d)$}: \\
We assume $(g)'$ is satisfied and  we prove $(e)'$. By Lemma~\ref{phistar}(v), for all $\alpha,\beta\in\N_0^d$, $\lambda,\mu>0$:
\beqsn
\|x^\beta D^\alpha u(x)\|_{L^p}\leq&&
C_\mu\|e^{\mu\omega(x)}e^{\mu\varphi^*\left(\frac{|\beta|}{\mu}\right)}D^\alpha u(x)\|_{L^p}\\
\leq&&
C_\mu e^{\mu\varphi^*\left(\frac{|\beta|}{\mu}\right)}
\|e^{\mu\omega(x)}D^\alpha u(x)\|_{L^p}\\
\leq&&C_{\mu,\lambda}e^{\mu\varphi^*\left(\frac{|\beta|}{\mu}\right)}
e^{\lambda\varphi^*\left(\frac{|\alpha|}{\lambda}\right)}
\eeqsn
for some $C_\mu, C_{\mu,\lambda}>0$.

Let us now assume $u\in\Sch_\omega(\R^d)$. Then condition $(g)$ of
Theorem~\ref{thSomega} is satisfied, and hence for $\sigma\geq (d+1)/bp$ and for every $\alpha\in\N_0^d$
and $\mu>0$:
\beqsn
\|e^{\mu\omega(x)}D^\alpha u(x)\|_{L^p}=&&
\|e^{(\mu+\sigma)\omega(x)}D^\alpha u(x) e^{-\sigma\omega(x)}\|_{L^p}\\
\leq&&\|e^{(\mu+\sigma)\omega(x)}D^\alpha u(x)\|_{L^\infty}\cdot\|e^{-\sigma\omega(x)}\|_{L^{p}}\\
\leq&&C_{\mu,\lambda}e^{\lambda\varphi^*\left(\frac{|\alpha|}{\lambda}\right)}
\eeqsn
for some $C_{\mu,\lambda}>0$ by $(g)$ and \eqref{add2}. \\[0.2cm]
\indent\underline{$(d)'\Leftrightarrow u\in\Sch_\omega(\R^d)$}: \\
Let $u\in\Sch_\omega(\R^d)$; then $u$ satisfies $(e)'$
for any $p$ (or $q$) in $[1,+\infty]$. Then $(d)'$ is trivially
satisfied for any $1\leq p,q\leq+\infty$.


In the opposite direction, we have that, using $(d)'(i)$ it is not difficult to see that (Lemma~\ref{phistar})
\beqsn
\vert\xi^\beta D^\alpha \hat{u}(\xi)\vert\leq
C_{\alpha,\lambda}
e^{\lambda\varphi^*\left(\frac{\vert\beta\vert}{\lambda}\right)},\ \ \xi\in\R^{d}.
\eeqsn
So $\hat{u}$ satisfies $(d)(ii)$ of Theorem \ref{thSomega}. In the same way, the fact that $u$ satisfies $(d)'(ii)$ implies that $\hat{u}$ satisfies $(d)(i)$
of Theorem \ref{thSomega}. Then $\hat{u}\in\Sch_\omega(\R^d)$. \\[0.2cm]
\indent\underline{$(h)'\Leftrightarrow u\in\Sch_\omega(\R^d)$}: \\
If $u\in\Sch_\omega(\R^d)$ then $u$ satisfies $(h)$ of Theorem \ref{thSomega}, and so
\beqsn
\Vert V_\psi u(z) e^{\lambda\omega(z)}\Vert_{L^{p,q}}\leq \Vert V_\psi u(z) e^{(\lambda+\sigma)\omega(z)}\Vert_{L^\infty}
\Vert e^{-\sigma\omega(z)}\Vert_{L^{p,q}}\leq C_\lambda,
\eeqsn
for $\sigma$ sufficiently large, from \eqref{add2}.

In the opposite direction, we prove that $(h)'\Rightarrow (e)$ of Theorem \ref{thSomega}. From the
proof of \cite[Proposition 2.10]{BJO2}, under condition \eqref{add1} instead of subadditivity,  we have
\beqsn
e^{-\lambda\varphi^*\left(\frac{\vert\alpha\vert}{\lambda}\right)}
e^{-\mu\varphi^*\left(\frac{\vert\beta\vert}{\mu}\right)} \vert y^\beta
D^\alpha u(y)\vert\leq C_{\lambda,\mu}\int_{\R^{2d}} \vert V_\psi u(z)\vert e^{(\mu L+3L\lambda+\sigma)\omega(z)} e^{-\sigma\omega(z)}dz
\eeqsn
for every $\sigma>0$; using H\"older's inequality for $L^{p,q}$ spaces we get
\beqsn
e^{-\lambda\varphi^*\left(\frac{\vert\alpha\vert}{\lambda}\right)}
e^{-\mu\varphi^*\left(\frac{\vert\beta\vert}{\mu}\right)} \vert y^\beta
D^\alpha u(y)\vert\leq C_{\lambda,\mu} \Vert V_\psi u(z)e^{(\mu L+3L\lambda+\sigma)\omega(z)} \Vert_{L^{p,q}}
\Vert e^{-\sigma\omega(z)}\Vert_{L^{p^\prime,q^\prime}}\leq C^\prime_{\lambda,\mu},
\eeqsn
if we choose $\sigma$ sufficiently large, cf. \eqref{add2}.
\end{proof}

We observe that Theorems~\ref{thSomega} and \ref{propSomega} provide equivalent systems of seminorms for the space $\Sch_\omega(\R^d)$.

\section{Real Paley-Wiener theorems for $\omega$-ultradifferentiable
functions}
\label{sec3}

Now, we prove different ``real Paley-Wiener theorems'' in the spirit
of \cite{B,A1,AD} in spaces of $\omega$-ultradifferentiable
functions. Moreover, we  analyze the behavior of time-frequency
representations (Gabor and Wigner) of $\omega$-ultradifferentiable functions which have Fourier transform with compact support.

We shall use in the following the notation $\langle f,g\rangle$ for the inner product in $L^2$ when $f,g\in L^2$, or (more generally) for the duality, that we consider as conjugate linear application of $f$ to $g$.

Here, we consider, for $R>0$ and a non-quasianalytic weight function
$\omega$, the space
\beqs
\label{PWR}
\PW_R^\omega(\R^d):=\Big\{f\in C^\infty(\R^d):\ \forall\lambda>0,\ \
\sup_{\alpha\in\N_0^d}\sup_{x\in\R^d}R^{-|\alpha|}
e^{\lambda\omega\left(\frac{x}{|\alpha|+1}\right)}|D^\alpha f(x)|<+\infty\Big\}.
\eeqs

\begin{Lemma}
\label{lemma1}
$\PW_R^\omega(\R^d)\subseteq\Sch_\omega(\R^d)$.
\end{Lemma}

\begin{proof}
Let $f\in \PW_R^\omega(\R^d)$ and let us first prove that $f\in\Sch(\R^d)$.
Indeed, there exists a constant $C>0$ such that for every $\alpha,\beta\in\N_0^d$ there exists $C_{\alpha,\beta}>0$ such that
\beqsn
|x^\beta D^\alpha f(x)|\leq&& C R^{|\alpha|}|x^\beta|
e^{-\omega\left(\frac{x}{|\alpha|+1}\right)}\\
\leq&&CR^{|\alpha|}|x|^{|\beta|}e^{-\frac{1}{D_{|\alpha|}}\omega(x)+1}\\
\leq&&C_\alpha R^{|\alpha|}e^{\frac{1}{D_{|\alpha|}}\varphi^*(|\beta| D_{|\alpha|})}=:C_{\alpha,\beta}\,,
\eeqsn
by \eqref{add3} and Lemma~\ref{phistar}(v).  Now, we prove conditions $(c)(i)$ and $(c)(ii)$ of {Theorem~\ref{thSomega}}. Condition $(c)(i)$ trivially follows from the definition of $\PW_R^\omega(\R^d)$ with
 $\alpha=0$. Let us prove condition $(c)(ii)$. For $|\xi|\geq1$ and $N\in\N_0$ we have:
\beqsn
|\hat{f}(\xi)|=&&\left|\int_{\R^d}f(x)e^{-i\langle x,\xi\rangle}dx\right|\\
\leq&&\frac{1}{|\xi|^{2N}}\left|\int_{\R^d}f(x)\Delta_x^Ne^{-i\langle x,\xi\rangle}dx\right|\\
\leq&&\frac{1}{|\xi|^{2N}}\int_{\R^d}|\Delta_x^Nf(x)|dx\\
\leq&&\frac{1}{|\xi|^{2N}}\sum_{{\vert\nu\vert}=N}\frac{N!}{{\nu!}}\int_{\R^d}
|D_x^{2{\nu}}f(x)|dx,
\eeqsn
where {$\nu\in\N^d$ and $D_x^{2\nu}=D_{x_1}^{2\nu_1}\cdots D_{x_d}^{2\nu_d}$}.
Since $f\in\PW_R^\omega(\R^d)$ we thus have, for $|\xi|\geq1$ and $\lambda\geq (d+1)/b$:
\beqs
\nonumber
|\hat{f}(\xi)|\leq&&
\frac{d^N}{|\xi|^{2N}}C_\lambda R^{2N}\int_{\R^d}
e^{-\lambda\omega\left(\frac{x}{2N+1}\right)}dx\\
\nonumber
=&&\frac{d^N}{|\xi|^{2N}}C_\lambda R^{2N}(2N+1)^d\int_{\R^d}
e^{-\lambda\omega(y)}dy\\
\label{4}
=&&C'_\lambda\frac{d^NR^{2N}(2N+1)^d}{|\xi|^{2N}}\\
\nonumber
\leq&&C'_\lambda\frac{(2^dR\sqrt{d})^{2N}}{|\xi|^{2N}}\\
\nonumber
\leq&&C_{\lambda'}|\xi|^{-2N}e^{\lambda'\varphi^*\left(\frac{2N}{\lambda'}\right)}
\eeqs
by Lemma~\ref{phistar}(viii), for some $C_\lambda,C'_\lambda,C_{\lambda'}>0$. Taking the infimum over $N\in\N_0$ and applying Lemma~\ref{phistar}(vi) we
have that, for all $\mu>0$ there exists $C_\mu>0$ such that for all $|\xi|\geq1$:
\beqsn
|\hat{f}(\xi)|\leq C_\mu e^{-\mu\omega(\xi)}.
\eeqsn
Since the above inequality is trivial for $|\xi|\leq1$, we finally have (c)(ii) and hence $f\in
\Sch_\omega(\R^d)$.
\end{proof}
In the following result, we denote by
\begin{equation}\label{qr}
Q_R:=\{\xi\in\R^d: |\xi|_\infty\leq R\},
\end{equation}
where  $\xi\in\R^d$ and $|\xi|_\infty$ is its sup norm.
\begin{Th}
\label{th1A1}
Let $R>0$ and $\omega$ a non-quasianalytic weight function.
The following conditions are equivalent:
\begin{itemize}
\item[(a)] The function  $f\in\PW_R^\omega(\R^d)$,

\item[(b)] The Fourier transform of $f$,
 $\hat{f}\in\D_{(\omega)}(\R^d)$ and
$\supp\hat{f}\subseteq  Q_R$.
\end{itemize}
\end{Th}

\begin{proof}
(a) $\Rightarrow$ (b). Let $f\in\PW_R^\omega(\R^d)$.
We integrate by parts,
\beqsn
\vert\hat{f}(\xi)\vert &=&\left\vert \frac{1}{\xi_1^{2N}+\dots +\xi_d^{2N}}\int_{\R^d}f(x)(D_{x_1}^{2N}+\dots +D_{x_d}^{2N})e^{-i\langle x,\xi\rangle}\,dx\right\vert \\
&\leq&\frac{1}{\xi_1^{2N}+\dots +\xi_d^{2N}}\sum_{j=1}^{{d}}
\int_{\R^d} \vert D_{x_j}^{2N} f(x)\vert\,dx.
\eeqsn
By hypothesis, we have that for every $\lambda>0$ there exists $C_\lambda$ such that
\beqs
\vert\hat{f}(\xi)\vert &\leq& C_\lambda \frac{1}{\xi_1^{2N}+\dots +\xi_d^{2N}}\sum_{j=1}^{{d}} R^{2N}\int_{\R^d} e^{-\lambda\omega(\frac{x}{2N+1})}\,dx \notag \\
&\leq& C_\lambda \frac{1}{\xi_1^{2N}+\dots +\xi_d^{2N}}{d} R^{2N}(2N+1)^d \int_{\R^d} e^{-\lambda\omega(y)}\,dy = D_{\lambda}\frac{{d} R^{2N}(2N+1)^d}{\xi_1^{2N}+\dots +\xi_d^{2N}}, \label{eq11}
\eeqs
for a constant $D_{\lambda}$ independent of $N$ and $\lambda\geq (d+1)/b$.
Now, we observe that for any $\xi\in\R^d$ such that $|\xi|_\infty>R$ we have $\sqrt[2N]{\xi_1^{2N}+\dots+\xi_d^{2N}}>R$, and so $\supp\hat{f}\subseteq Q_R$.

\vskip.5\baselineskip


\vskip.5\baselineskip

(b) $\Rightarrow$ (a) Suppose that $\hat{f}\in\D_{(\omega)}(\R^d)\subset\Sch_\omega(\R^d)$
with $\supp\hat{f}\subseteq Q_R$.
By Fourier inversion formula in $\Sch(\R^d)$, for {$x\neq 0$} and $N\in\N_0$:
\beqs
\nonumber
|D^\alpha f(x)|=&&\frac{1}{(2\pi)^d}\left|\int_{\R^d}\F(D^\alpha f)(\xi)e^{i\langle x,\xi\rangle}
d\xi\right|\\
\nonumber
\leq&&\left|\int_{\R^d}\xi^\alpha\hat{f}(\xi)\frac{1}{|x|^{2N}}\Delta_\xi^Ne^{i\langle x,\xi\rangle}
d\xi\right|\\
\nonumber
\leq&&\frac{1}{|x|^{2N}}\int_{\R^d}|\Delta_\xi^N\xi^\alpha\hat{f}(\xi)|d\xi\\
\nonumber
\leq&&\frac{1}{|x|^{2N}}\int_{\R^d}\sum_{{\vert\nu\vert=N}}
\frac{N!}{{\nu!}}\Big|D_{\xi_1}^{{2\nu_1}}\cdots D_{\xi_d}^{{2\nu_d}}
\Big(\xi_1^{\alpha_1}\cdots\xi_d^{\alpha_d}\hat{f}(\xi)\Big)\Big|d\xi\\
\nonumber
 \leq&&\frac{1}{|x|^{2N}}\sum_{{\vert\nu\vert=N}}
\frac{N!}{{\nu!}}
\sum_{h_1=0}^{{\min\{2\nu_1,\alpha_1\}}}\binom{{2\nu_1}}{h_1}
\cdots\sum_{h_d=0}^{{\min\{2\nu_d,\alpha_d\}}}\binom{{2\nu_d}}{h_d}\\
\label{4bis}
&&
\frac{\alpha_1!}{(\alpha_1-h_1)!}\cdots\frac{\alpha_d!}{(\alpha_d-h_d)!}
\int_{|\xi|_\infty\leq R}|\xi_1^{\alpha_1-h_1}\cdots\xi_d^{\alpha_d-h_d}|\cdot
\left|D_\xi^{{2\nu-h}}\hat{f}(\xi)\right|d\xi,
\eeqs
where we denoted {$D_\xi^{2\nu-h}=D_{\xi_1}^{2\nu_1-h_1}
 \cdots D_{\xi_d}^{2\nu_d-h_d}$}.
 Since $\hat{f}\in\Sch_\omega(\R^d)$,  there exists
 $C_{\mu,\lambda}>0$ such that, applying {Theorem~\ref{thSomega}(g)} in \eqref{4bis},
 for $|x|\geq1$ and $N\in\N_0$:
 \beqsn
|D^\alpha f(x)|\leq&&
 \frac{1}{|x|^{2N}}\sum_{{\vert\nu\vert=N}}
\frac{N!}{{\nu!}}
\sum_{h_1=0}^{{\min\{2\nu_1,{\alpha_1}\}}}\binom{{2\nu_1}}{h_1}
\cdots\sum_{h_d=0}^{{\min\{2\nu_d,{\alpha_d}\}}}\binom{{2\nu_d}}{h_d}\\
&&\cdot\alpha_1^{h_1}\cdots\alpha_d^{h_d}
\int_{|\xi|_\infty\leq R}{\vert\xi_1\vert^{\alpha_1-h_1}\dots\vert\xi_d\vert^{\alpha_d-h_d}} \left|D_\xi^{{2\nu-h}}\hat{f}(\xi)\right|d\xi\\
\leq&&\frac{1}{|x|^{2N}}
\sum_{{\vert\nu\vert=N}}
\frac{N!}{{\nu!}}
\sum_{h_1=0}^{{\min\{2\nu_1,{\alpha_1}\}}}\binom{{2\nu_1}}{h_1}
\cdots\sum_{h_d=0}^{{\min\{2\nu_d,{\alpha_d}\}}}\binom{{2\nu_d}}{h_d}\\
&&\cdot|\alpha|^{2N}\left(1+\frac1R\right)^{2N}
R^{|\alpha|}\int_{\R^d}C_{\mu,\lambda}e^{\lambda\varphi^*\left(\frac{2N-|h|}{\lambda}\right)}
e^{-\mu\omega(\xi)}d\xi\\
\leq&& C_\lambda\frac{1}{|x|^{2N}}d^{N}2^{2N}(|\alpha|+1)^{2N}
\left(1+\frac1R\right)^{2N}e^{\lambda\varphi^*\left(\frac{2N}{\lambda}\right)}R^{|\alpha|},
 \eeqsn
 for some $C_\lambda>0$, where we have fixed $\mu\geq (d+1)/b$.

Taking the infimum over $N\in\N_0$ and applying Lemma~\ref{phistar}(vi)
we have therefore,
for {$|x|\geq 2\sqrt{d}(\vert\alpha\vert+1)(1+\frac{1}{R})$},
\beqs
\label{6}
|D^\alpha f(x)|\leq C_\lambda e^{-\left(\lambda-\frac2b\right)
\omega\left(\frac{x}{2\sqrt{d}(|\alpha|+1)\left(1+\frac1R\right)}\right)-\frac{2a}{b}}R^{|\alpha|}
\eeqs
for $a\in\R,b>0$ as in condition $(\gamma)$ of Definition \ref{defomega}.

{Let us consider now $|x|< 2\sqrt{d}(\vert\alpha\vert+1)(1+\frac{1}{R})$. We have
\beqsn
\vert D^\alpha f(x)\vert &=& \frac{1}{(2\pi)^d} \left\vert \int_{\R^d} \mathcal{F}(D^\alpha f)(\xi) e^{i\langle x,\xi\rangle}\,d\xi\right\vert \\
&\leq&\int_{Q_R} \vert\xi_1\vert^{\alpha_1}\dots\vert\xi_d\vert^{\alpha_d}\vert\hat{f}(\xi)\vert\,d\xi \\
&\leq& CR^{\vert\alpha\vert},
\eeqsn
for $C=\Vert\hat{f}\Vert_{L^1(\R^d)}$. Since $\omega$ is increasing we have that \eqref{6} is true also for $|x|< 2\sqrt{d}(\vert\alpha\vert+1)(1+\frac{1}{R})$, for a constant $C_\lambda$ which depends on $\lambda,a,b,R,d$ and $\omega(1)$. By \eqref{6} and \eqref{add3} we
}
finally have that for every $\lambda'>0$
there exists $C_{\lambda'}>0$, depending on $\omega, \lambda', d, R, a$ and $b$, such that
\beqsn
|D^\alpha f(x)|\leq C_{\lambda'}e^{-\lambda'\omega\left(\frac{x}{|\alpha|+1}\right)}R^{|\alpha|}.
\eeqsn
This proves that $f\in\PW_R^\omega(\R^d)$.
\end{proof}


Let us define, for a function $g$ on $\R^d$:
\beqs
R_g:=\sup\{|x|_\infty:\ x\in\supp g\}.
\eeqs

The next result treats two different cases: the first one does not need weight functions and it is a natural extension of Theorem 1 of \cite{B} for several variables; in the other case, we assume two different additional conditions  on the weight function: subadditivity (condition \eqref{add4}) or a ``mild'' condition introduced in \cite{BMM} that guarantees that the weight does not increase too slowly (condition \eqref{bmm}). We shall use in the following the notation $f^{(\alpha)}$ for $D^\alpha f$.
\begin{Prop}
\label{th3A1}
Let $1\leq p\leq+\infty$ and $f\in C^\infty(\R^{d})$. We have:
\begin{enumerate}
\item If  $f^{(\alpha)}(x)\in L^p(\R^{d})$ for all $\alpha\in\N_0^{d}$, we have
\beqs
\label{stella0}
\lim_{n\to+\infty}\left(\max_{|\alpha|=n}\left\|f^{(\alpha)}(x)\right\|_{L^p}\right)^{1/n}
=R_{\hat f}.
\eeqs

\item Assume that $e^{\lambda\omega\left(\frac{x}{|\alpha|+1}\right)}f^{(\alpha)}(x)\in L^p(\R^{d})$ for all $\alpha\in\N_0^{d}$ and
for some $\lambda>0$, and that the weight function $\omega$ satisfies one of the following conditions:
\begin{enumerate}
\item It is sub-additive, i.e.,
\beqs
\label{add4}
\omega(t_1+t_2)\leq\omega(t_1)+\omega(t_2),\quad  t_1,t_2\geq 0;
\eeqs
 \item There is a constant $H>1$ such that
\beqs
\label{bmm}
2\omega(t)\leq \omega(H t)+H,\quad t\geq 0.
\eeqs

\end{enumerate}
Then
\beqs
\label{stella}
\lim_{n\to+\infty}\left(\max_{|\alpha|=n}\left\|e^{\mu\omega\left(\frac{x}{|\alpha|+1}\right)}f^{(\alpha)}(x)\right\|_{L^p}\right)^{1/n}
=R_{\hat f}, \mbox{ for all }\ 0\le \mu\le \lambda.
\eeqs
\end{enumerate}
\end{Prop}

\begin{Rem}
\begin{em}
We observe that, in general, $R_{\hat f}\in\{t\in\R;\ t\geq0\}\cup\{+\infty\}$, so that $\hat f$
may not have compact support. Moreover, the limit \eqref{stella} does not depend on $\mu$.
\end{em}
\end{Rem}

\begin{proof}[Proof of Proposition~\ref{th3A1}]
It suffices to see $(2)$, since $(1)$ can be proved in the same way (it is statement $(2)$ for $\lambda=0$). We can assume that $p<\infty$, because the same proof is valid for $p=\infty$ with some small modifications. First, we consider $\phi\in\Sch_\omega(\R^{d})$ such that $\hat\phi$ has
compact support.
Then, by Theorem~\ref{th1A1}, we have that
$\phi\in\PW_{R_{\hat\phi}}^\omega(\R^{d})$ and hence, for every
$1\leq p<+\infty$, $\lambda> 0$, and  $\sigma\geq 2/bp$:
\beqsn
\left\|e^{\lambda\omega\left(\frac{x}{|\alpha|+1}\right)}\phi^{(\alpha)}(x)\right\|_{L^p}^{1/n}=&&
\left\|e^{(\lambda+\sigma)\omega\left(\frac{x}{|\alpha|+1}\right)}\phi^{(\alpha)}(x)
e^{-\sigma\omega\left(\frac{x}{|\alpha|+1}\right)}\right\|_{L^p}^{1/n}\\
\leq&&\left\|e^{(\lambda+\sigma)\omega\left(\frac{x}{|\alpha|+1}\right)}\phi^{(\alpha)}(x)\right\|_{L^\infty}^{1/n}
\cdot\left\|e^{-\sigma\omega\left(\frac{x}{|\alpha|+1}\right)}\right\|_{L^p}^{1/n}\\
\leq&&(C_{\lambda+\sigma} R_{\hat\phi}^{|\alpha|})^{1/n}(|\alpha|+1)^{\frac{d}{pn}}
\|e^{-\sigma\omega(x)}\|_{L^p}^{1/n}.
\eeqsn
So, if we take the maximum when $|\alpha|=n$ and then the limit when $n$ tends to infinity, we deduce
\beqs
\label{301}
\limsup_{n\to+\infty}\left(\max_{|\alpha|=n}\left\|e^{\lambda\omega\left(\frac{x}{|\alpha|+1}\right)}\phi^{(\alpha)}(x)
\right\|_{L^p}^{1/n}\right)\leq R_{\hat\phi}\,, \qquad
 p\in[1,+\infty),\ \lambda>0.
\eeqs

Now, we consider $f\in C^\infty(\R^{d})$ such that
$e^{\lambda\omega\left(\frac{x}{|\alpha|+1}\right)}f^{(\alpha)}(x)\in L^p(\R^{d})$ for all $\alpha\in\N_0^{d}$.
We observe that $f\in\Sch'(\R^{d})$ and hence its Fourier transform is well
defined. Assume, for the moment, that
$\supp\hat f$ is compact, so that $R_{\hat f}\in\R$.

We observe that if the weight satisfies hypothesis $(2) (a)$, i.e., it is sub-additive, we have
\beqs\label{sub-adit}
\lambda\omega\left(\frac{x}{n+1}\right)\le \lambda\omega\left(\frac{y}{n+1}\right)+\lambda\omega\left(\frac{x-y}{n+1}\right),
\eeqs
for any $x,y\in\R^{d},$ $\lambda\ge 0$ and $n\in\N.$ On the other hand, it is easy to deduce from hypothesis $(2)(b)$ that for each $k\in\N,$
\beqs\nonumber
2^{k}\omega(x)\le \omega(H^{k}x)+H(2^{k-1}+2^{k-2}+\cdots+1), \quad x\in\R^{d},
\eeqs
and hence, $\omega(x)\le 2^{-k}\omega(H^{k}x)+H$, for all $x\in\R^{d}.$ Now, we take $k\in\N$ so that $L\le 2^{k}$, where $L\ge 1$ is the constant of condition $(\alpha)$ of Definition~\ref{defomega}. Then, we select $n\in\N$ big enough with $H^{k}\le n+1$ to deduce, from \eqref{add1},
\beqs\label{BMM-consequence}\nonumber
\lambda\omega\left(\frac{x}{n+1}\right)\le && \lambda L\omega\left(\frac{y}{n+1}\right)+\lambda L\omega\left(\frac{x-y}{n+1}\right)+\lambda L\\ \nonumber
\le && \lambda L\omega\left(\frac{y}{n+1}\right)+\lambda L 2^{-k}\omega\left(\frac{H^{k}(x-y)}{n+1}\right)+\lambda L+H\\
\le && \lambda L\omega\left(\frac{y}{n+1}\right)+\lambda \omega(x-y)+\lambda L+H,
\eeqs
for all $x,y\in\R^{d}$. Hence, under both hypotheses on the weight function $\omega$, we have, by \eqref{sub-adit} or \eqref{BMM-consequence}, for each $x,y\in\R^{d}$ and $n$ big enough,
\beqs\label{hypotheses}
\lambda\omega\left(\frac{x}{n+1}\right)\le \lambda L\omega\left(\frac{y}{n+1}\right)+\lambda \omega(x-y)+D_{\lambda},
\eeqs
for some constant $D_{\lambda}$ that depends on $\lambda\ge 0$ and the weight function $\omega.$

Let $\varepsilon>0$ and choose $\phi\in\Sch_\omega(\R^{d})$ such that $\hat{\phi}\equiv1$
in a neighborhood of $[-R_{\hat f},R_{\hat f}]^{d}$ and $\hat{\phi}\equiv0$
outside $[-R_{\hat f}-\varepsilon,R_{\hat f}+\varepsilon]^{d}$. Then
$\hat{f}=\hat{f}\cdot\hat{\phi}$ and hence, by the properties of the Fourier transform, $f=f\ast\phi$. Now, by \eqref{hypotheses}, we obtain
\beqs
\nonumber
&&\limsup_{n\to+\infty}\left(\max_{|\alpha|=n}\left\|e^{\lambda\omega\left(\frac{x}{|\alpha|+1}\right)}f^{(\alpha)}(x)
\right\|_{L^p}^{1/n}\right)=
\limsup_{n\to+\infty}\left(\max_{|\alpha|=n}\left\|e^{\lambda\omega\left(\frac{x}{n+1}\right)}
{f}\ast{\phi}^{(\alpha)}(x)
\right\|_{L^p}^{1/n}\right)\\
\nonumber
&&\qquad =\limsup_{n\to+\infty}\left(\max_{|\alpha|=n}\left\|
e^{\lambda\omega\left(\frac{x}{n+1}\right)}\int_{\R^d} {\phi}^{(\alpha)}(y){f}(x-y)dy\right\|_{L^p}^{1/n}\right)\\
\label{add5}
 &&\qquad \le \limsup_{n\to+\infty}e^{D_{\lambda}/n}\left(\max_{|\alpha|=n}
\left\|\int_{\R^d}{\phi}^{(\alpha)}(y)e^{\lambda L\omega\left(\frac{y}{n+1}\right)}{f}(x-y)
e^{\lambda\omega(x-y)}dy\right\|_{L^p}^{1/n}\right)\\
\nonumber
&&\qquad \le \limsup_{n\to+\infty}\left(\max_{|\alpha|=n}\left\|e^{\lambda L\omega\left(\frac{x}{n+1}\right)}
\phi^{(\alpha)}(x)\right\|_{L^1}^{1/n}\right)\left\|
e^{\lambda\omega(x)}f(x)\right\|_{L^p}^{1/n} \le R_{\hat\phi}\leq R_{\hat f}+\varepsilon,
\eeqs
since, by assumption, $e^{\lambda\omega(x)} f(x)\in L^p(\R^{d})$ and, by the construction of $\phi$, $R_{\hat\phi}\leq R_{\hat f}+\varepsilon$. Now, as $\varepsilon>0$ is arbitrary, we obtain
\beqs
\label{302}
\limsup_{n\to+\infty}\left(\max_{|\alpha|=n}\left\|e^{\lambda\omega\left(\frac{x}{n+1}\right)}f^{(\alpha)}(x)
\right\|_{L^p}^{1/n}\right)\leq R_{\hat f}.
\eeqs
We remark that when $\supp\hat f$ is not compact, $R_{\hat{f}}=+\infty$ and, in this case, \eqref{302} is still valid.

Take now $0\neq\xi^0\in\supp\hat f$, and assume w.l.o.g. that $0<\varepsilon<|\xi^0_{1}|=|\xi^{0}|_{\infty}$, where $\xi^{0}=(\xi^{0}_{1},\ldots,\xi^{0}_{d})\in \R^{d}$. We take
$\psi\in\D_{(\omega)}(\R^{d})$ with $\Pi_{1}\supp\psi\subseteq\left[\xi^0_{1}-\frac\varepsilon2,
\xi^0_{1}+\frac\varepsilon2\right]$ and $\langle\hat{f},\psi\rangle\neq0$, where $\Pi_{1}:\R^{d}\to \R$ is the projection in the first variable.

Then, for $\xi\in\R^{d}$ with $\xi_{1}\neq 0$, $\lambda>0$ and $1\leq p<+\infty$ we have:
\beqsn
\lefteqn{(|\xi^0_{1}|-\varepsilon)^n|\langle\hat{f}(\xi),\psi(\xi)\rangle|=
(|\xi^0_{1}|-\varepsilon)^n|\langle\xi_{1}^n\hat{f}(\xi),\xi_{1}^{-n}\psi(\xi)\rangle|}\\
&&\ \ =(|\xi^0_{1}|-\varepsilon)^n|\langle\widehat{D_{1}^nf}(\xi),\xi_{1}^{-n}\psi(\xi)\rangle|\\
&&\ \ =(|\xi^0_{1}|-\varepsilon)^n|\langle{D^n_{1}f}(x),\F^{-1}(\xi^{-n}_{1}\psi(\xi))(x)\rangle|\\
&&\ \ \le (\vert\xi^0_{1}\vert-\varepsilon)^n \Vert D_{1}^{n}f\Vert_{L^p} \Vert \mathcal{F}^{-1} (\xi^{-n}_{1} \psi(\xi))\Vert_{L^{p'}}\\
&&\ \ \le  (\vert\xi_{1}^0\vert-\varepsilon)^n \Vert e^{\lambda\omega\left(\frac{x}{n+1}\right)} D^{n}_{1}f(x)\Vert_{L^p}
\left\Vert \frac{1}{(1+x_{1}^2+\cdots+x_{d}^{2})^{d}}\mathcal{F}^{-1}
\left[(1-\Delta_\xi)^{d}(\xi^{-n}_{1}\psi(\xi))\right]\right\Vert_{L^{p'}}\\
&&\ \ \le  (\vert\xi^0_{1}\vert-\varepsilon)^n \Vert e^{\lambda\omega\left(\frac{x}{n+1}\right)} D^{n}_{1}f(x)\Vert_{L^p}
\left\Vert\frac{1}{(1+x_{1}^2+\cdots+x_{d}^{2})^{d}}\right\Vert_{L^{p'}} \left\Vert \mathcal{F}^{-1} \left[(1-\Delta_\xi)^{d}(\xi_{1}^{-n}\psi(\xi))\right]\right\Vert_{L^\infty}.
\eeqsn
We have
\beqsn
\lefteqn{ \mathcal{F}^{-1} \left[(1-\Delta_\xi)^{d}(\xi^{-n}_{1}\psi(\xi))\right]=\sum_{|\nu|=\nu_{1}+\cdots+\nu_{d+1}=d}\frac{d!}{\nu!} \mathcal{F}^{-1}\left[D^{2\nu_{1}}_{\xi_{1}}\cdots D^{2\nu_{d}}_{\xi_{d}}(\xi^{-n}_{1}\psi(\xi))\right]}\\
&&\ \ \ = \sum_{|\nu|=d}\frac{d!}{\nu!}\sum_{h=0}^{2\nu_{1}}\binom{{2\nu_1}}{h}i^{h}\frac{(n+h-1)!}{(n-1)!} \mathcal{F}^{-1}\left[\xi^{-n-h}_{1} D^{2\nu_{1}-h}_{\xi_{1}}\cdots D^{2\nu_{d}}_{\xi_{d}}\psi(\xi)\right].
\eeqsn

Therefore, we obtain
\beqsn
\lefteqn{\left\Vert \mathcal{F}^{-1} \left[(1-\Delta_\xi)^{d}(\xi^{-n}_{1}\psi(\xi))\right]\right\Vert_{L^\infty}}\\
&&\le 4^{d} (d+1)^{d}(n+2d)^{2d} \max_{|\nu|=d,0\le h \le 2\nu_{1}}
\Vert \xi_{1}^{-n-h}D^{2\nu_{1}-h}_{\xi_{1}}\cdots D^{2\nu_{d}}_{\xi_{d}}\psi(\xi)\Vert_{L^1}\\
&&\leq\frac{1}{(\vert\xi_{1}^0\vert-\varepsilon/2)^n} 4^{d} (d+1)^{d}(n+2d)^{2d} \max_{|\nu|=d,0\le h \le 2\nu_{1}}
\Vert \xi_{1}^{-h}D^{2\nu_{1}-h}_{\xi_{1}}\cdots D^{2\nu_{d}}_{\xi_{d}}\psi(\xi)\Vert_{L^1}.
\eeqsn
We then obtain
\beqsn
\lefteqn{(\vert\xi^0_{1}\vert-\varepsilon)\vert\langle \hat{f},\psi\rangle\vert^{1/n}}\\
&&\le \frac{\vert\xi^0_{1}\vert-\varepsilon}{\vert\xi^0_{1}\vert-\varepsilon/2}
\Vert e^{\lambda\omega\left(\frac{x}{n+1}\right)} D^{n}_{1}f(x)\Vert_{L^p}^{1/n}
\left\Vert\frac{1}{(1+x_{1}^2+\cdots+x_{d}^{2})^{d}}\right\Vert_{L^{p'}}^{1/n} (n+2d)^{2d/n}C(\psi)^{1/n},
\eeqsn
for a constant $C(\psi)$ that depends on $\psi$, the support of $\psi$ and its partial derivatives up to the order $2d$, and the dimension $d$. Hence, since $\frac{\vert\xi_{1}^0\vert-\varepsilon}{\vert\xi^0_{1}\vert-\varepsilon/2}\leq 1$,
\beqsn
|\xi^0|_{\infty}-\varepsilon\leq\liminf_{n\to+\infty}
\left(\max_{|\alpha|=n}\left\|e^{\lambda\omega\left(\frac{x}{n+1}\right)}f^{(\alpha)}(x)
\right\|_{L^p}^{1/n}\right)
\leq\limsup_{n\to+\infty}\left(\max_{|\alpha|=n}\left\|e^{\lambda\omega\left(\frac{x}{n+1}\right)}f^{(\alpha)}(x)
\right\|_{L^p}^{1/n}\right)\leq R_{\hat f}
\eeqsn
by \eqref{302}.

By the arbitrariness of $\varepsilon>0$ and then of
$\xi^0\in\supp\hat f$:
\beqsn
R_{\hat f}\leq\liminf_{n\to+\infty}\left(\max_{|\alpha|=n}\left\|e^{\lambda\omega\left(\frac{x}{n+1}\right)}f^{(\alpha)}(x)
\right\|_{L^p}^{1/n}\right)
\leq\limsup_{n\to+\infty}\left(\max_{|\alpha|=n}\left\|e^{\lambda\omega\left(\frac{x}{n+1}\right)}f^{(\alpha)}(x)
\right\|_{L^p}^{1/n}\right)
\leq R_{\hat f},
\eeqsn
and, hence, there exists
\beqsn
\lim_{n\to+\infty}\left(\max_{|\alpha|=n}\left\|e^{\lambda\omega\left(\frac{x}{n+1}\right)}f^{(\alpha)}(x)
\right\|_{L^p}^{1/n}\right)=R_{\hat f},
\eeqsn
for $\lambda>0$ and $1\leq p<+\infty$.
\end{proof}

\begin{Rem}\label{add6}
The condition
$e^{\lambda\omega\left(\frac{x}{|\alpha|+1}\right)}f^{(\alpha)}(x)\in L^p$ for all $\lambda\ge 0$ is equivalent to
$e^{\lambda\omega(x)}f^{(\alpha)}(x)\in L^p$ for all $\lambda\ge 0$ by \eqref{add3}. Therefore, if in Proposition \ref{th3A1} we ask that $e^{\lambda\omega(x)} f^{(\alpha)}(x)\in L^p(\R^{d})$
for all $\alpha\in\N_0^{d}$ and all $\lambda\geq 0$, \eqref{stella} is true without the additional assuptions
\eqref{add4} or \eqref{bmm}. Indeed, in \eqref{add5} we can use \eqref{add1} directly.
\end{Rem}

As we have already mentioned, Proposition \ref{th3A1} in the case $\lambda=0$ is \cite[Theorem 1]{B} for several variables, cf. \cite[Theorem 3]{A1} also. On the other hand, we are interested in the case $\lambda>0$ in order to get Paley-Wiener theorems for ultradifferentiable functions; see Theorem \ref{th4A2} below. To this aim, first we prove that, under the assumptions of Proposition~\ref{th3A1}, if \eqref{stella} is satisfied for some $R_{\hat f}\in\R$ and for all $\lambda>0$, then $u\in\Sch_\omega(\R^{d})$. We need some lemmas.



\begin{Lemma}
\label{lemma310}
Let $f\in C^\infty(\R^{d})$ such that $e^{\lambda\omega(x)}f^{(\alpha)}(x)
\in L^p(\R^{d})$ for all $\alpha\in\N_0^{d}$, $\lambda>0$, and some $1\leq p\leq +\infty$. Then $f\in\Sch(\R^{d})$.
\end{Lemma}

\begin{proof}
Since $f\in\Sch'(\R^{d})$, we can apply the Fourier
transform to $f$. We fix $\alpha,\beta\in\N_0^{d}$ and choose $\lambda>0$ big enough such that
$x^{\beta-\gamma} e^{-\lambda\omega\left(\frac{x}{|\alpha-\gamma|+1}\right)}\in L^{p'}(\R^d)$,  for every
$\gamma\leq\min\{\alpha,\beta\}$ and for $1/p+1/p'=1$, and we apply H\"older's inequality to obtain
\beqsn
|\xi^\alpha D^\beta\hat{f}(\xi)|=&&|\F\big(D_x^\alpha(x^\beta f(x))\big)(\xi)|
\leq\int_{\R^{d}}
|D_x^\alpha(x^\beta f(x))|dx\\
\leq&&\sum_{\afrac{\gamma\leq\alpha}{\gamma\leq\beta}}\binom\alpha\gamma
\frac{\beta!}{(\beta-\gamma)!}\int_{\R^{d}}|x^{\beta-\gamma} D_x^{\alpha-\gamma} f(x)|dx\\
=&&\sum_{\afrac{\gamma\leq\alpha}{\gamma\leq\beta}}\binom\alpha\gamma
\frac{\beta!}{(\beta-\gamma)!}\int_{\R^{d}}|
e^{\lambda\omega\left(\frac{x}{|\alpha-\gamma|+1}\right)}D_x^{\alpha-\gamma} f(x)|
\cdot |x^{\beta-\gamma}e^{-\lambda\omega\left(\frac{x}{|\alpha-\gamma|+1}\right)}|dx\\
\leq && \sum_{\afrac{\gamma\leq\alpha}{\gamma\leq\beta}}\binom\alpha\gamma
\frac{\beta!}{(\beta-\gamma)!}\|e^{\lambda\omega\left(\frac{x}{|\alpha-\gamma|+1}\right)}
f^{(\alpha-\gamma)}(x)\|_{L^p}
\cdot \|x^{\beta-\gamma}e^{-\lambda\omega\left(\frac{x}{|\alpha-\gamma|+1}\right)}\|_{L^{p'}}\leq C_{\alpha,\beta,\lambda},
\eeqsn
which finishes the proof.
\end{proof}



\begin{Lemma}
\label{lemma308}
Let $1\leq p\leq +\infty$ and $f\in C^\infty(\R^{d})$ with
$e^{\lambda\omega(x)}f^{(\alpha)}(x)
\in L^p(\R^{d})$ for all $\alpha\in\N_0^{d}$ and for all $\lambda>0$.
If $\hat f$ has compact support, we have
\beqs
\label{304}
\sup_{\xi\in\R^{d}}e^{\lambda\omega(\xi)}|\hat{f}(\xi)|<+\infty,
\qquad\mbox{ for all }\lambda>0.
\eeqs
\end{Lemma}

\begin{proof}
Assume that $\xi=(\xi_{1},\ldots,\xi_{d})\neq0$, and that $|\xi|_{\infty}=|\xi_{1}|$. Given $n\in\N_0$ and $\lambda\geq \frac{d+1}{bp'}$, where $1/p+1/p'=1$, we can write
\beqsn
|\hat{f}(\xi)|=&&\frac{1}{|\xi_{1}|^n}\left|\int_{\R^d} f(x)D_{x_1}^ne^{-ix\xi}dx\right|\\
\leq&&\frac{1}{|\xi_{1}|^n}\int_{\R^d} e^{\lambda\omega\left(\frac{x}{n+1}\right)}|D_{1}^{n}f(x)|
e^{-\lambda\omega\left(\frac{x}{n+1}\right)}dx\\
\leq&&\frac{1}{|\xi_{1}|^n}\left\|e^{\lambda\omega\left(\frac{x}{n+1}\right)}D^{n}_{1}f(x)
\right\|_{L^p}
\cdot\left\|e^{-\lambda\omega\left(\frac{x}{n+1}\right)}\right\|_{L^{p'}}\\
=&&\frac{C_\lambda}{|\xi_{1}|^n}(n+1)^{\frac{d}{p'}}
\left\|e^{\lambda\omega\left(\frac{x}{n+1}\right)}D^{n}_{1}f(x)\right\|_{L^p}
\eeqsn
for $C_\lambda=\|e^{-\lambda\omega(x)}\|_{L^{p'}}<+\infty$ from \eqref{add2}.

Since $\hat f$ has compact support by assumption, by Proposition~\ref{th3A1} and Remark~\ref{add6}, we have that
\eqref{stella} is satisfied with $R_{\hat f}\in\R$. Therefore, there exists a constant
$D\in\R$, depending only on $f$, such that, for all $n\in\N_{0}$,
\beqsn
\|e^{\lambda\omega\left(\frac{x}{n+1}\right)}D^{n}_{1}f(x)\|_{L^p}
\leq D^n,
\eeqsn
and hence, by Lemma~\ref{phistar}(viii),
\beqs
\label{303}
|\hat{f}(\xi)|\leq C_\lambda\frac{D^n}{|\xi_{1}|^n}(n+1)^{\frac{d}{p'}}
\leq C_\lambda\frac{\tilde{D}^n}{|\xi_{1}|^n}n!
\leq C'_\lambda|\xi_{1}|^{-n}e^{\lambda\varphi^*\left(\frac n\lambda\right)}
\eeqs
for some $\tilde{D},\ C'_\lambda>0$.

Now, by Lemma~\ref{phistar}(vi), if we assume $|\xi_{1}|\geq1$,
\beqsn
|\hat{f}(\xi)|\leq C'_\lambda e^{-\left(\lambda-\frac1b\right)\omega(\xi_{1})-\frac ab}.
\eeqsn
Hence, it is suffices to take $\lambda>1/b$ big enough to finish the proof.
\end{proof}

\begin{Lemma}
\label{lemma312}
Let $1\leq p\leq +\infty$ and $f\in C^\infty(\R^{d})$ such that
$e^{\lambda\omega(x)}f^{(\alpha)}(x)\in L^p(\R)$ for all $\alpha\in\N_0^{d}$
and $\lambda>0$. If $\hat f$ has compact support, then $f\in\Sch_\omega(\R^{d})$.
\end{Lemma}

\begin{proof}
By Lemmas \ref{lemma310} and \ref{lemma308}
we have that $f\in\Sch(\R^{d})$ and, for every $\lambda>0$, there exists $C_\lambda>0$
such that
$\|e^{\lambda\omega(\xi)}\hat{f}(\xi)\|_{L^\infty}\leq C_\lambda.$
Moreover
$\|e^{\lambda\omega(x)}f(x)\|_{L^p}\leq C'_\lambda$
for some $C'_\lambda>0$ by assumption. It follows, from Theorem~\ref{propSomega}$\,(c)'$ with $q=\infty$, that $f\in\Sch_\omega(\R^{d})$.
\end{proof}

\subsection{Relation with the Wigner transform}
Proposition~\ref{th3A1} proves that the radius of the support of $\hat f$
can be computed with the limit
\eqref{stella} for any $\lambda\geq0$.
Now, we give a characterization of the support of $\hat f$ in terms of the Wigner
transform. First, we introduce the following real Paley-Wiener space defined by means of the Gabor transform:
\begin{Def}
\label{defPWG}
Let $T,R>0$ and define, for $\psi\in\PW_T^\omega(\R^d)$,
\beqsn
\PWG_R^{\omega,\psi}(\R^d):=\{&&f\in C^\infty(\R^d)\cap\Sch'_\omega(\R^d):
\ \mbox{for each }\lambda,\mu>0,\\
&&\sup_{N\in\N_0}\sup_{x,\xi\in\R^d}(R+T)^{-N}
\frac{1}{(N+1)^{d/2}}e^{\lambda\omega\left(\frac{x}{N+1}\right)+\mu\omega(\xi)}
|\xi|_\infty^N|V_\psi f(x,\xi)|<+\infty\}.
\eeqsn
\end{Def}

\begin{Prop}
\label{propG1}
Let $\psi\in\PW_T^\omega(\R^d)$. Then
\beqsn
\PW_R^\omega(\R^d)\subseteq\PWG_R^{\omega,\psi}(\R^d).
\eeqsn
\end{Prop}

\begin{proof}
Let $f\in\PW_R^\omega(\R^d)$. Fix $\xi\in\R^d\setminus\{0\}$.
Then $|\xi|_\infty=|\xi_j|$ for some $1\leq j\leq d$ and hence
\beqsn
|\xi|_\infty^N|V_\psi f(x,\xi)|=&&|\xi_j^NV_\psi f(x,\xi)|
=\left|\xi_j^N\int_{\R^d}f(y)\overline{\psi(y-x)}e^{-i\langle y,\xi\rangle}dy\right|\\
=&&\left|\int_{\R^d}f(y)\overline{\psi(y-x)}D_{y_j}^N(e^{-i\langle y,\xi\rangle})dy\right|\\
=&&\left|\int_{\R^d}D_{y_j}^N\big(f(y)\overline{\psi(y-x)}\big)e^{-i\langle y,\xi\rangle}dy\right|\\
\leq&&\sum_{k=0}^N\binom Nk
\int_{\R^d}|D_{y_j}^kf(y)|\cdot|D_{y_j}^{N-k}\psi(y-x)|dy.
\eeqsn

Since $f\in\PW_R^\omega(\R^d)$ and $\psi\in\PW_T^\omega(\R^d)$, it is not difficult to see
that for every $\lambda>0$ there is $C_{\lambda}>0$ such that
\beqs
\label{308}
|\xi|_\infty^N|V_\psi f(x,\xi)|\leq C_{\lambda}(R+T)^N
e^{-\lambda\omega\left(\frac{x}{N+1}\right)}(N+1)^d,
\eeqs
for all $x,\xi\in\R^d$ and $N\in\N_0$.

Moreover, since $f,\psi\in\Sch_\omega(\R^d)$ by Lemma~\ref{lemma1}, then
 $V_\psi f\in\Sch_\omega(\R^{2d})$ also (\cite[Thm. 2.7]{GZ}) and, hence, for all
$\mu>0$ there exists $C_\mu>0$ such that
\beqs
\label{309}
|V_\psi f(x,\xi)|\leq C_\mu e^{-\mu\omega(\xi)},
\eeqs
since $\omega(x,\xi)\geq\omega(\xi)$.

By Theorem~\ref{th1A1} we have that $\supp\hat{f}\subseteq Q_R$,
$\supp\hat{\psi}\subseteq Q_T$ and hence the projection on $\xi$ of the support
of $V_\psi f$ satisfies
\beqs
\label{310}
\Pi_\xi\left(\supp V_\psi f(x,\xi)\right)\subseteq Q_{R+T},\quad x\in\R^{d},
\eeqs
as it can be deduced for example from \cite[formula (3.8)]{G}.
From \eqref{309} and \eqref{310} we have that
\beqs
\label{311}
|\xi|_\infty^N|V_\psi f(x,\xi)|\leq C_\mu(R+T)^Ne^{-\mu\omega(\xi)}
\eeqs
for all $x,\xi\in\R^d$, $N\in\N_0$ and $\mu>0$.

Combining \eqref{308} and \eqref{311} we finally have:
\beqsn
|\xi|_\infty^N|V_\psi f(x,\xi)|=&&\sqrt{\big(|\xi|_\infty^N|V_\psi f(x,\xi)|\big)^2}\\
\leq&&\sqrt{C_{\lambda}(R+T)^Ne^{-\lambda\omega\left(\frac{x}{N+1}\right)}
(N+1)^dC_\mu(R+T)^Ne^{-\mu\omega(\xi)}}\\
\leq&&C_{\lambda,\mu}(R+T)^N(N+1)^{d/2}
e^{-\frac{\lambda}{2}\omega\left(\frac{x}{N+1}\right)}
e^{-\frac\mu2\omega(\xi)}
\eeqsn
for some $C_{\lambda,\mu}>0$ and for all $x,\xi\in\R^d$, $N\in\N_0$, $\lambda,\mu>0$. Therefore $f\in\PWG_R^{\omega,\psi}(\R^d)$.
\end{proof}

Given the space defined in \eqref{40}, we have the following result:
\begin{Prop}
\label{prop1}
Let $f,\psi\in\Sch_\omega(\R^d)$ and $p,q\in[1,+\infty]$.
Then, for every $\lambda,\mu\geq0$,
\beqs
\label{312}
\limsup_{N\to+\infty}\left\|e^{\lambda\omega\left(\frac{x}{N+1}\right)+\mu\omega(\xi)}
|\xi|_\infty^N V_\psi f(x,\xi)\right\|_{L^{p,q}}^{1/N}\leq R_{\hat f}+R_{\hat\psi}.
\eeqs
\end{Prop}

\begin{proof}
If $\supp\hat f$ or $\supp\hat\psi$ are not compact, then $R_{\hat f}=+\infty$ or,
respectively, $R_{\hat\psi}=+\infty$, so the inequality \eqref{312} is trivial.
So, we can assume that $\supp\hat{f}$ and $\supp\hat{\psi}$ are compact, and hence
$R_{\hat f},R_{\hat\psi}\in\R$.
By Theorem~\ref{th1A1} and Proposition~\ref{propG1}, we have
$f\in\PW_{R_{\hat f}}^\omega(\R^d)\subseteq\PWG_{R_{\hat f}}^{\omega,\psi}(\R^d)$
and hence for $\sigma$ and $\tau$ sufficiently large, from \eqref{add2} we obtain
\beqsn
\lefteqn{\limsup_{N\to+\infty}\left\|e^{\lambda\omega\left(\frac{x}{N+1}\right)+\mu\omega(\xi)}
|\xi|_\infty^N V_\psi f(x,\xi)\right\|_{L^{p,q}}^{1/N}}\\
&&\le
\limsup_{N\to+\infty}\left\|e^{(\lambda+\sigma)\omega\left(\frac{x}{N+1}\right)+(\mu+\tau)\omega(\xi)}
|\xi|_\infty^N V_\psi f(x,\xi)\right\|_{L^\infty}^{1/N}
\cdot\left\|e^{-\sigma\omega\left(\frac{x}{N+1}\right)-\tau\omega(\xi)}\right\|_{L^{p,q}}^{1/N}\\
&&\le \limsup_{N\to+\infty}C_{\lambda,\mu}^{\frac1N}(R_{\hat f}+R_{\hat \psi})
(N+1)^{\frac{d}{2N}+\frac{d}{pN}}\|e^{-\sigma\omega(x)-\tau\omega(\xi)}\|_{L^{p,q}}^{1/N}=R_{\hat f}+R_{\hat\psi},
\eeqsn
for some $C_{\lambda,\mu}>0$, if $p<+\infty$.
If $p=+\infty$ the proof is similar.
\end{proof}

We introduce now the following notation for the translation and modulation operators; for
$x,\xi,x^0,\xi^0\in\R^d$ we denote
\beqsn
T_{x^0} f(x) = f(x-x^0),\qquad M_{\xi^0} f(\xi)=e^{i\langle \xi^0,\xi\rangle}f(\xi).
\eeqsn

\begin{Ex}
\label{remstrict}
\begin{em}
The inequality \eqref{312} is strict, in general.
Let us consider, for instance, $f\in\Sch_\omega(\R)$ with $\supp\hat{f}\subseteq
[R_{\hat f}-\mu,R_{\hat f}]$ for some $0<\mu<R_{\hat f}<+\infty$.
Then
\beqs
\label{330}
\||\xi|_\infty^N V_f f(x,\xi)\|_{L^{p,q}}^{1/N}
\leq\mu\|V_f f(x,\xi)\|_{L^{p,q}}^{1/N}
\eeqs
since
\beqsn
\Pi_\xi\supp V_f f(x,\xi)=\union\limits_{x\in\R}\supp(\hat{f}\ast M_{-x}\bar{\tilde{\hat{f}}})(\xi)
\subseteq[R_{\hat f}-\mu,R_{\hat f}]+[-R_{\hat f},-R_{\hat f}+\mu]=[-\mu,\mu],
\eeqsn
where $\tilde{f}(x)=f(-x)$,
by \cite[Lemma 3.1.1]{G}. Since $\|V_f f(x,\xi)\|_{L^{p,q}}$ does not depend on $N$, letting $N\to+\infty$ in
\eqref{330} we get that
\beqsn
\limsup_{N\to+\infty}\||\xi|_\infty^N V_f f(x,\xi)\|_{L^{p,q}}^{1/N}
\leq\mu<R_{\hat f}<2R_{\hat f}.
\eeqsn
\end{em}
\end{Ex}
On the other hand, for the right choice of the window function we get the equality in \eqref{312}, as the next result shows. This fact becomes crucial for the analysis of real Paley-Wiener theorems in terms of the Wigner transform. In the next result the number $R_{\hat{f}}$ could be $+\infty.$

\begin{Prop}
\label{prop2}
Let $f\in\Sch_\omega(\R^d)$ and $p,q\in[1,+\infty]$. Then, for all $\lambda,\mu\geq0$, we have
\beqs
\label{318}
\lim_{N\to+\infty}\left\|
e^{\lambda\omega\left(\frac{x}{N+1}\right)+\mu\omega(\xi)}
|\xi|_\infty^N V_{\tilde f} f(x,\xi)\right\|_{L^{p,q}}^{1/N}=2R_{\hat f}.
\eeqs
\end{Prop}

\begin{proof}
By Proposition~\ref{prop1} we have
\beqs
\label{317}
\limsup_{N\to+\infty}\left\|e^{\lambda\omega\left(\frac{x}{N+1}\right)+\mu\omega(\xi)}
|\xi|_\infty^N V_{\tilde f} f(x,\xi)\right\|_{L^{p,q}}^{1/N}
\leq 2R_{\hat f},
\eeqs
since $\hat{\tilde{f}}(\xi)=\hat{f}(-\xi)$ and hence $R_{\hat{\tilde{f}}}=R_{\hat f}\,$.

Now, we fix $\xi^0\in\supp\hat f$ and $0<\varepsilon<2|\xi^0|_\infty$,  choose $\phi_\varepsilon,\psi_\varepsilon\in\Sch_\omega(\R^d)$ with
$\supp\hat\phi_\varepsilon,\hat\psi_\varepsilon\subseteq Q_{\varepsilon/4}$ and
\beqs
\label{314}
\langle\hat f,T_{\xi^0}\hat\psi_\varepsilon\rangle\neq0,\quad
\langle\hat f,T_{\xi^0}\hat{\tilde{\phi}}_\varepsilon\rangle\neq0.
\eeqs

Note that, by \cite[formula (3.10)]{G},
\beqsn
\Pi_\xi\supp\left(V_{\phi_\varepsilon}(M_{2\xi^0}\psi_\varepsilon)(x,\xi)\right)
=&&\Pi_\xi\supp\left(e^{-i\langle x,\xi\rangle}V_{\hat{\phi}_\varepsilon}
(\widehat{M_{2\xi^0}\psi_\varepsilon})(\xi,-x)\right)\\
=&&\Pi_\xi\supp\left(V_{\hat{\phi}_\varepsilon}
(T_{2\xi^0}\hat{\psi}_\varepsilon)(\xi,-x)\right)\\
=&&\Pi_\xi\supp\left(\bigl(T_{2\xi^0}\hat{\psi}_\varepsilon*
M_{-x}\widehat{\overline{\phi}}_\varepsilon\bigr)(\xi)\right)\\
\subseteq&&
Q_{\frac\varepsilon2}(2\xi^0):=
\{\xi\in\R^d:\ |\xi-2\xi^0|_\infty\leq\varepsilon/2\}.
\eeqsn

Then, for $\xi\in\Pi_\xi\supp\left(V_{\phi_\varepsilon}(M_{2\xi^0}\psi_\varepsilon)\right)$
we have $|\xi|_\infty\geq2|\xi^0|_\infty-\varepsilon/2$. Hence, for all
$\xi\in\R^d\setminus\{0\}$:
\beqs
\nonumber
\left\|e^{-\lambda\omega\left(\frac{x}{N+1}\right)-\mu\omega(\xi)}
|\xi|_\infty^{-N} V_{\phi_\varepsilon}(M_{2\xi^0}\psi_\varepsilon)(x,\xi)\right\|_{L^{p',q'}}
\leq&&\left(2|\xi^0|_\infty-\frac\varepsilon2\right)^{-N}\left\|
V_{\phi_\varepsilon}(M_{2\xi^0}\psi_\varepsilon)\right\|_{L^{p',q'}}\\
\label{313}
\leq&& C_{\varepsilon,\xi^0}\left(2|\xi^0|_\infty-\frac\varepsilon2\right)^{-N}
\eeqs
for some $C_{\varepsilon,\xi^0}>0$, since
$V_{\phi_\varepsilon}(M_{2\xi^0}\psi_\varepsilon)\in\Sch_\omega(\R^{2d})
\subseteq L^{p',q'}(\R^{2d})$,
where we have denoted by $p',q'$ the conjugate exponents of $p$ and $q$
respectively.

On the other hand, for $\xi\in\R^d\setminus\{0\}$, by \eqref{313}:
\beqsn
\lefteqn{(2|\xi^0|_\infty-\varepsilon)^N|\langle V_{\tilde f}f(x,\xi),e^{-i\langle\xi^0,x\rangle}
V_{\phi_\varepsilon}(M_{2\xi^0}\psi_\varepsilon)(x,\xi)\rangle|}\\
&&\le(2|\xi^0|_\infty-\varepsilon)^N\left|\langle
e^{\lambda\omega\left(\frac{x}{N+1}\right)+\mu\omega(\xi)}
|\xi|_\infty^N V_{\tilde f} f(x,\xi),e^{-\lambda\omega\left(\frac{x}{N+1}\right)-\mu\omega(\xi)}
|\xi|_\infty^{-N} e^{-i\langle\xi^0,x\rangle}
V_{\phi_\varepsilon}(M_{2\xi^0}\psi_\varepsilon)(x,\xi)\rangle\right|\\
&&\le (2|\xi^0|_\infty-\varepsilon)^N\left\|
e^{\lambda\omega\left(\frac{x}{N+1}\right)+\mu\omega(\xi)}
|\xi|_\infty^N V_{\tilde f} f\right\|_{L^{p,q}}\cdot
\left\|e^{-\lambda\omega\left(\frac{x}{N+1}\right)-\mu\omega(\xi)}
|\xi|_\infty^{-N}V_{\phi_\varepsilon}(M_{2\xi^0}\psi_\varepsilon)\right\|_{L^{p',q'}}\\
&&\le C_{\varepsilon,\xi^0}
\frac{(2|\xi^0|_\infty-\varepsilon)^N}{(2|\xi^0|_\infty-\varepsilon/2)^N}
\left\|
e^{\lambda\omega\left(\frac{x}{N+1}\right)+\mu\omega(\xi)}
|\xi|_\infty^N V_{\tilde f} f\right\|_{L^{p,q}}\le C_{\varepsilon,\xi^0}
\left\|
e^{\lambda\omega\left(\frac{x}{N+1}\right)+\mu\omega(\xi)}
|\xi|_\infty^N V_{\tilde f} f\right\|_{L^{p,q}}.
\eeqsn
Then
\beqs
\notag
(2|\xi^0|_\infty-\varepsilon)\left|\langle V_{\tilde f}f(x,\xi),e^{-i\langle\xi^0,x\rangle}
V_{\phi_\varepsilon}(M_{2\xi^0}\psi_\varepsilon)(x,\xi)\rangle\right|^{\frac1N}
\leq \\
\label{315}
\leq C_{\varepsilon,\xi^0}^{\frac1N}
\left\|
e^{\lambda\omega\left(\frac{x}{N+1}\right)+\mu\omega(\xi)}
|\xi|_\infty^N V_{\tilde f} f\right\|_{L^{p,q}}^{\frac1N}.
\eeqs

Let us now remark that
\beqsn
V_{\phi_\varepsilon}(M_{2\xi^0}\psi_\varepsilon)(x,\xi)=&&
\int_{\R^d}e^{i\langle 2\xi^0,y\rangle}\psi_\varepsilon(y)\overline{\phi_\varepsilon(y-x)}
e^{-i\langle y,\xi\rangle}dy\\
=&&e^{i\langle\xi^0,x\rangle}\int_{\R^d}e^{i\langle\xi^0,y\rangle}\psi_\varepsilon(y)
\overline{e^{-i\langle\xi^0,y-x\rangle}\phi_\varepsilon(y-x)}e^{-i\langle y,\xi\rangle}dy\\
=&&e^{i\langle\xi^0,x\rangle}V_{M_{-\xi^0}\phi_\varepsilon}
(M_{\xi^0}\psi_\varepsilon)(x,\xi)
\eeqsn
and therefore, from \cite[Thm. 3.2.1]{G}, by \eqref{314}:
\beqsn
\lefteqn{\langle V_{\tilde f}f(x,\xi),e^{-i\langle\xi^0,x\rangle}V_{\phi_\varepsilon}(M_{2\xi^0}\psi_\varepsilon)(x,\xi)
\rangle=
\langle V_{\tilde f}f, V_{M_{-\xi^0}\phi_\varepsilon}(M_{\xi^0}\psi_\varepsilon)\rangle}\\
&&=\langle f,M_{\xi^0}\psi_\varepsilon\rangle\cdot\overline{\langle\tilde{f},
M_{-\xi^0}\phi_\varepsilon\rangle}
=\langle \widehat{f},\widehat{M_{\xi^0}\psi_\varepsilon}\rangle\cdot\overline{\langle
\widehat{\tilde{f}},\widehat{M_{-\xi^0}\phi_\varepsilon}\rangle}\\
&&=\langle \widehat{f},T_{\xi^0}\widehat{\psi}_\varepsilon\rangle\cdot\overline{\langle
\widehat{\tilde{f}},T_{-\xi^0}\widehat{\phi}_\varepsilon\rangle}
=\langle \widehat{f},T_{\xi^0}\widehat{\psi}_\varepsilon\rangle\cdot\overline{\langle
\widehat{{f}},T_{\xi^0}\widehat{\tilde{\phi}}_\varepsilon\rangle}
\neq0.
\eeqsn
Therefore
\beqsn
\lim_{N\to+\infty}\left|\langle V_{\tilde f}f(x,\xi),e^{-i\langle\xi^0,x\rangle}
V_{\phi_\varepsilon}(M_{2\xi^0}\psi_\varepsilon)(x,\xi)\rangle\right|^{1/N}=1
\eeqsn
and from \eqref{315} we obtain that
\beqs
\label{316}
(2|\xi^0|_\infty-\varepsilon)\leq
\liminf_{N\to+\infty}\left\|
e^{\lambda\omega\left(\frac{x}{N+1}\right)+\mu\omega(\xi)}
|\xi|_\infty^N V_{\tilde f} f\right\|_{L^{p,q}}^{1/N}.
\eeqs
By the arbitrariness of $0<\varepsilon<2|\xi^0|_\infty$ and of
$\xi^0\in\supp\hat f$, from \eqref{316} and \eqref{317}, we finally 
obtain \eqref{318}.
\end{proof}

\begin{Cor}
\label{corWig1}
Let $f\in\Sch_\omega(\R^d)$ and $p,q\in[1,+\infty]$. Then, for all $\lambda,\mu\geq0$:
\beqs
\label{320}
\lim_{N\to+\infty}\left\|
e^{\lambda\omega\left(\frac{x}{N+1}\right)+\mu\omega(\xi)}
|\xi|_\infty^N \Wig f(x,\xi)\right\|_{L^{p,q}}^{1/N}=
R_{\hat f}.
\eeqs
\end{Cor}

\begin{proof}
By \cite[Lemma 4.3.1]{G}, if $p,q\in[1,+\infty)$:
\beqs
\notag
\lefteqn{\left\|e^{\lambda\omega\left(\frac{x}{N+1}\right)+\mu\omega(\xi)}
|\xi|_\infty^N \Wig f(x,\xi)\right\|_{L^{p,q}}}\\
\notag
&&=\left\|e^{\lambda\omega\left(\frac{x}{N+1}\right)+\mu\omega(\xi)}
|\xi|_\infty^N 2^d e^{2i\langle x,\xi\rangle} V_{\tilde f}f(2x,2\xi)\right\|_{L^{p,q}}\\
\label{320a}
&&=2^d 2^{-\frac dp}2^{-\frac dq}
\left\|e^{\lambda\omega\left(\frac{y}{2(N+1)}\right)+\mu\omega\left(\frac{\eta}{2}\right)}
\left|\frac\eta2\right|_\infty^N  V_{\tilde f}f(y,\eta)\right\|_{L^{p,q}}.
\eeqs
Using the fact that $\omega$ is increasing and satisfies condition $(\alpha)$
of Definition~\ref{defomega} we have
\beqsn
\frac{1}{L}\omega(t)-1\leq\omega\left(\frac{t}{2}\right)\leq \omega(t),
\eeqsn
and so, by \eqref{320a},
\beqsn
&&2^{\frac{d}{N}(1-\frac{1}{p}-\frac{1}{q})}\frac{1}{2} e^{-\lambda/N-\mu/N}
\left\|e^{\frac{\lambda}{L}\omega\left(\frac{y}{N+1}\right)+\frac{\mu}{L}\omega\left(\eta\right)}
\left|\eta\right|_\infty^N V_{\tilde f}f(y,\eta)\right\|_{L^{p,q}}^{1/N}\leq
\left\|e^{\lambda\omega\left(\frac{x}{N+1}\right)+\mu\omega(\xi)}
|\xi|_\infty^N \Wig f(x,\xi)\right\|_{L^{p,q}}^{1/N}\\
&&\qquad\qquad\leq 2^{\frac{d}{N}(1-\frac{1}{p}-\frac{1}{q})}\frac{1}{2}
\left\|e^{\lambda\omega\left(\frac{y}{N+1}\right)+\mu\omega\left(\eta\right)}
\left|\eta\right|_\infty^N  V_{\tilde f}f(y,\eta)\right\|_{L^{p,q}}^{1/N}.
\eeqsn

Consequently, from Proposition~\ref{prop2}, we deduce
\beqsn
\lim_{N\to+\infty}\left\|
e^{\lambda\omega\left(\frac{x}{N+1}\right)+\mu\omega(\xi)}
|\xi|_\infty^N \Wig f(x,\xi)\right\|_{L^{p,q}}^{1/N}
=\frac12\cdot2R_{\hat f}=R_{\hat f}\,,
\eeqsn
for $1\leq p,q<\infty$.
If $p$ and/or $q$ is $\infty$ the proof is similar.
\end{proof}

\begin{Cor}
\label{corWig2}
Let $f\in\Sch_\omega(\R^d)$ and $p,q\in[1,+\infty]$. Then
\beqs
\label{321}
&&\lim_{N\to+\infty}\left\||\xi|_\infty^N\Wig f(x,\xi)\right\|_{L^{p,q}}^{1/N}=R_{\hat f}\\
\label{322}
&&\lim_{N\to+\infty}\left\||x|_\infty^N\Wig f(x,\xi)\right\|_{L^{p,q}}^{1/N}=R_f.
\eeqs
\end{Cor}

\begin{proof}
Formula \eqref{321} follows from \eqref{320} with $\lambda=\mu=0$.

Formula \eqref{322} follows from \cite[Prop. 4.3.2]{G} and \eqref{321}
applied to $\hat f$:
\beqsn
\lim_{N\to+\infty}\left\||x|_\infty^N\Wig f(x,\xi)\right\|_{L^{p,q}}^{1/N}=&&
\lim_{N\to+\infty}\left\||x|_\infty^N\Wig \hat{f}(\xi,-x)\right\|_{L^{p,q}}^{1/N}\\
=&&\lim_{N\to+\infty}\left\||x|_\infty^N\Wig \hat{f}(\xi,x)\right\|_{L^{p,q}}^{1/N}
=R_{\hat{\hat{f}}}=R_f.
\eeqsn
\end{proof}

If we consider formula \eqref{321} for $p=q=2$ in the one-dimensional case,  the multiplication by $|\xi|^N$ cannot be replaced by the derivatives $D_x^N$
of the Wigner transform of a real valued function $f\in\Sch_\omega(\R)$.
Indeed, if we denote by
\beqsn
Af(x,\xi):=\int_{\R}f\left(t+\frac x2\right)
\overline{f\left(t-\frac x2\right)}e^{-i\langle t,\xi\rangle}dt
\eeqsn
the {\em ambiguity function} $Af$ of $f$, by \cite[Lemma 4.3.4]{G}, we obtain
\beqs
\nonumber
\|D_x^N\Wig f(x,\xi)\|_{L^2}=&&\|\F\big(D_x^N\Wig f(x,\xi)\big)\|_{L^2}\\
\label{340}
=&&\|y^N\widehat{\Wig f}(y,\eta)\|_{L^2}
=\|y^NAf(-\eta,y)\|_{L^2}.
\eeqs
Now, since $f$ is real valued by assumption,
\beqsn
Af(-\eta,y)=&&\int_{\R}f\left(t-\frac\eta2\right)\overline{f\left(t+\frac\eta2\right)}
e^{-i\langle t,y\rangle}dt\\
=&&\int_{\R}f(u-\eta){f(u)}e^{-i\langle u-\frac\eta2,y\rangle}du
=e^{\frac i2\langle\eta,y\rangle} V_{{f}}{f}(\eta,y).
\eeqsn
Hence, by \eqref{340},
\beqsn
\limsup_{N\to +\infty}\|D_x^N\Wig f(x,\xi)\|_{L^2}^{1/N}=
\limsup_{N\to +\infty}\|y^N V_f f(\eta,y)\|_{L^2}^{1/N},
\eeqsn
which can be strictly smaller than $R_{\hat f}$, by Example~\ref{remstrict}.

On the other hand, if $f=\tilde f$, by Proposition~\ref{prop2},
\beqsn
\lim_{N\to +\infty}\|D_x^N\Wig f(x,\xi)\|_{L^2}^{1/N}=
\lim_{N\to +\infty}\|y^N V_{\tilde f} f(\eta,y)\|_{L^2}^{1/N}=2R_{\hat f}>R_{\hat f}.
\eeqsn

\begin{Lemma}
\label{lemWig10}
Let $p,q\in[1,+\infty]$ and $f\in\Sch^\prime_\omega(\R^d)$ such that
\beqsn
e^{\lambda\omega(x)+\mu\omega(\xi)} \Wig f(x,\xi)\in L^{p,q}(\R^{2d})
\eeqsn
for all $\lambda,\mu>0$. Then $f\in\Sch_\omega(\R^d)$.
\end{Lemma}

\begin{proof}

We observe that $\Wig f\in L^1(\R^{2d})$, since
\beqsn
\Vert \Wig f\Vert_{L^1}\leq \Vert e^{\lambda\omega(x)+\mu\omega(\xi)}\Wig f(x,\xi)\Vert_{L^{p.q}}\Vert e^{-\lambda\omega(x)-\mu\omega(\xi)}\Vert_{L^{p^\prime,q^\prime}}<\infty
\eeqsn
by hypothesis and \eqref{add2}, provided that $\lambda\geq (d+1)/bp'$ and $\mu\geq (d+1)/bq'$.
Then, by applying the inverse partial Fourier transform with respect to $\xi$ to $\Wig f(x,\xi)$ we get
\beqs
\label{4001}
f\left(x+\frac{t}{2}\right)\overline{f\left(x-\frac{t}{2}\right)} = (2\pi)^{-d}\int_{\R^d}
 \Wig f(x,\xi) e^{i\langle\xi,t\rangle}d\xi.
\eeqs
Then, the element $f\left(x+\frac{t}{2}\right)\overline{f\left(x-\frac{t}{2}\right)}$, that a priori belongs to $\Sch_\omega^\prime(\R^{2d}_{(t,x)})$, is in fact a function in $L^\infty(\R^d_t)$ for almost every $x\in\R^d$, and is in $L^1(\R^d_x)$ for every $t\in\R^d$. Now, suppose that $f\not\equiv 0$ (otherwise the result is trivial), and let $\phi_0\in\Sch_\omega(\R^d)$ such that $\langle \overline{f},\phi_0\rangle\neq 0$. For a function $\phi\in\Sch_\omega(\R^d)$, consider
\beqsn
\Phi(t,x) = \phi\left(x+\frac{t}{2}\right) \phi_0\left( x-\frac{t}{2}\right) \in\Sch_\omega(\R^{2d}),
\eeqsn
and apply the two distributions in \eqref{4001} to the test function $\Phi$; on the right-hand side we can write the application as an integral, and then we obtain
\beqsn
\langle f,\phi\rangle\langle \overline{f},\phi_0\rangle = (2\pi)^{-d} \int_{\R^{2d}}\left(\int_{\R^d} \Wig f(x,\xi) e^{i\langle\xi,t\rangle}d\xi\right) \phi\left(x+\frac{t}{2}\right) \phi_0\left( x-\frac{t}{2}\right) dx\,dt.
\eeqsn
Then by the change of variables $x+t/2 = y$, $x-t/2=s$ and by Fubini Theorem we obtain
\beqsn
\langle f,\phi\rangle = \frac{1}{(2\pi)^d \langle \overline{f},\phi_0\rangle} \int_{\R^d}
\left( \int_{\R^{2d}}\Wig f\left(\frac{y+s}{2},\xi\right) e^{i\langle\xi,y-s\rangle} \phi_0(s)\,d\xi\,ds\right) \phi(y)\,dy,
\eeqsn
and so we get that $f$ is a function in $L^1(\R^d)$ given by
\beqs
\label{41}
f(x) = \frac{1}{(2\pi)^d \langle \overline{f},\phi_0\rangle} \int_{\R^{2d}}\Wig f\left(\frac{x+s}{2},\xi\right) e^{i\langle\xi,x-s\rangle} \phi_0(s)\,d\xi\,ds.
\eeqs
In order to prove that $f\in\Sch_\omega(\R^d)$ we shall prove that $f$ satisfies condition $(c)'$ of Theorem~\ref{propSomega}. Suppose that $p<+\infty$. By \eqref{41} and Minkowski inequality, cf. for example \cite[6.19]{F}, we have
\beqsn
\Vert e^{\lambda\omega(x)} f(x)\Vert_{L^p} &\leq& \frac{1}{(2\pi)^d |\langle \overline{f},\phi_0\rangle|} \left(\int_{\R^d}\left(\int_{\R^{2d}} e^{\lambda\omega(x)} \left\vert \Wig f\left(\frac{x+s}{2},\xi\right) \right\vert \vert\phi_0(s)\vert d\xi ds\right)^p dx\right)^{1/p} \\
&\leq& \frac{1}{(2\pi)^d |\langle \overline{f},\phi_0\rangle|} \int_{\R^{2d}}\left(\int_{\R^d} \left(e^{\lambda\omega(x)} \left\vert \Wig f\left(\frac{x+s}{2},\xi\right) \right\vert \vert\phi_0(s)\vert \right)^p dx\right)^{1/p} d\xi ds.
\eeqsn
Writing $C_0 = ((2\pi)^d |\langle \overline{f},\phi_0\rangle|)^{-1}$, using H\"older inequality in the $\xi$-integral and \eqref{add1} we obtain, for $\mu\geq (d+1)/bq'$,
\beqs
\notag
\Vert e^{\lambda\omega(x)} f(x)\Vert_{L^p} &\leq& C_0 e^{\lambda L} \int_{\R^{2d}} e^{-\mu\omega(\xi)}e^{\mu\omega(\xi)}\\
\notag
&&\cdot
 \left(\int_{\R^d} \left( e^{\lambda L\omega(x+s)}\left\vert\Wig f\left(\frac{x+s}{2},\xi\right)\right\vert\right)^p dx\right)^{1/p} e^{\lambda L\omega(s)}\vert\phi_0(s)\vert d\xi ds \\
\notag
&\leq& C_0 2^{d/p} e^{\lambda (L^2+L)} \Vert e^{-\mu\omega(\xi)}\Vert_{L^{q^\prime}}\\
\notag
&&\cdot\left(
 \int_{\R^d} e^{\lambda L\omega(s)}\vert\phi_0(s)\vert ds\right) \left\Vert e^{\mu\omega(\xi)} \left( \int_{\R^d}\left( e^{\lambda L^2\omega(y)}\vert \Wig f(y,\xi)\vert \right)^p dy\right)^{1/p} \right\Vert_{L^q(\R^d_\xi)}\\
\label{42}
&=&C_\lambda\Vert e^{\lambda L^2\omega(y)+\mu\omega(\xi)}\Wig f(y,\xi)\Vert_{L^{p.q}} <\infty
\eeqs
by hypothesis and \eqref{add2}. In the case $p=+\infty$ the same proof works, with small modifications, so \eqref{42} holds for every $p$ and $q$.

Now, let $\phi_1\in\Sch_\omega(\R^d)$ be such that $\langle\overline{\hat f},\phi_1\rangle \neq 0$, and $q<+\infty$. We apply \eqref{41}  to $\hat f$ and use \cite[Prop. 4.3.2]{G} to get
\beqsn
\Vert e^{\lambda\omega(\xi)}\hat f(\xi)\Vert_{L^q}&\leq& C_1\left( \int_{\R^d} \left(e^{\lambda\omega(\xi)}\int_{\R^{2d}} \left\vert\Wig \hat{f}\left(\frac{\xi+s}{2},y\right) \right\vert \vert\phi_1(s)\vert dy ds\right)^q d\xi\right)^{1/q}\\
&=&C_1\left( \int_{\R^d} \left(e^{\lambda\omega(\xi)}\int_{\R^{2d}} \left\vert\Wig f\left(-y,\frac{\xi+s}{2}\right) \right\vert \vert\phi_1(s)\vert dy ds\right)^q d\xi\right)^{1/q},
\eeqsn
where $C_1 = ((2\pi)^{d}|\langle\overline{\hat f},\phi_1\rangle|)^{-1}$. We apply the change of variables $(\xi+s)/2=\eta$ in the $\xi$-integral, $-y=x$, use condition $(\alpha)$ of Definition~\ref{defomega} and  \eqref{add1} and  H\"older's inequality in the $x$-integral to obtain, for $\mu\geq (d+1)/bp'$,
\beqs
\notag
\Vert e^{\lambda\omega(\xi)}\hat f(\xi)\Vert_{L^q}&\leq& C_1 2^{\frac dq} e^{\lambda(L^2+L)} \int_{\R^d} e^{\lambda L\omega(s)}\vert\phi_1(s)\vert ds \\
\notag
&&\cdot\left( \int_{\R^d}\left( e^{\lambda L^2\omega(\eta)}\int_{\R^d} e^{-\mu\omega(x)}e^{\mu\omega(x)} \vert\Wig f(x,\eta)\vert dx\right)^q d\eta\right)^{1/q} \\
\notag
&\leq&C_\lambda\left(\int_{\R^d}\left( \int_{\R^d} \left( e^{\lambda L^2\omega(\eta)}e^{\mu\omega(x)}\vert\Wig f(x,\eta)\vert\right)^p dx\right)^{q/p} d\eta\right)^{1/q}\\
\label{43}
&=&C_\lambda \Vert e^{\mu\omega(x)+\lambda L^2\omega(\eta)}\Wig f(x,\eta)\Vert_{L^{p,q}}<\infty,
\eeqs
by hypothesis, where $C_\lambda=C_1 2^{d/q} e^{\lambda(L^2+L)} \int e^{\lambda L\omega(s)}\vert\phi_1(s)\vert ds\Vert e^{-\mu\omega(x)}\Vert_{L^{p^\prime}}<+\infty$ by \eqref{add2}. If $q=\infty$ the same proof works, with small modifications, so \eqref{43} holds for every $p$ and $q$. By \eqref{42} and \eqref{43} the function $f$
satisfies Theorem~\ref{propSomega}\ $(c)'$.
\end{proof}


We can now prove the following theorem that, besides the classical result in ultradifferentiable classes (see \cite{Bj,BMT,Fie}), contains real ultradifferentiable Paley-Wiener theorems in the spirit of \cite{A2} and a new equivalent condition on the Wigner transform. Given $R>0$, for the compact set $Q_{R}$, as defined in \eqref{qr}, we denote its supporting function $H_{Q_{R}}(x):=\sup_{y\in Q_{R}}\langle x,y\rangle$ simply by $H_{R}(x)$ for all $x\in\R^{d}.$

\begin{Th}
\label{th4A2}
Let $1\leq p,q\leq +\infty$ and $R>0$.
Then the following conditions are equivalent:
\begin{itemize}
\item[(a)]
$f$ is an entire function in $\C^{d}$ and for all $k\in\N_0$ there exists $C_k>0$ such that
\beqsn
|f(z)|\leq C_k e^{H_{R}(\Im z)-k\omega(z)},\qquad z\in\C^{d}.
\eeqsn
\item[(b)]
$f\in\PW_R^\omega(\R^{d})$.
\item[(c)]
$f\in C^\infty(\R^{d})$,
$e^{\lambda\omega(x)}f^{(\alpha)}(x)\in L^p(\R^{d})$ for all $\alpha\in\N_0^{d}$ and
$\lambda\geq0$ and
\beqs
\label{duestelle}
\lim_{n\to+\infty}\left(\max_{|\alpha|=n}\left\|e^{\lambda\omega\left(\frac{x}{n+1}\right)}f^{(\alpha)}(x)\right\|_{L^p}\right)^{1/n}
=R_{\hat{f}}\le R.
\eeqs
\item[(d)]
$f\in\Sch_\omega(\R^{d})$ and
$
\supp\hat{f}\subseteq Q_R.
$
\item[(e)]
$f\in\Sch'_\omega(\R^{d})$,
$e^{\lambda\omega(x)+\mu\omega(\xi)} \Wig f(x,\xi)\in L^{p,q}(\R^{2d})$
for all $\lambda,\mu\geq0$ and
\beqs
\label{trestelle}
\lim_{n\to +\infty}\left\Vert e^{\lambda\omega\left(\frac{x}{n+1}\right)+\mu\omega(\xi)} \vert\xi\vert^n_{\infty}\Wig f(x,\xi)\right\Vert_{L^{p,q}}^{1/n}=R_{\hat{f}}\le R.
\eeqs
\end{itemize}
\end{Th}

\begin{proof}
{$(a)\Leftrightarrow(d)$}:
This is Paley-Wiener theorem in $\D_{(\omega)}(\R^{d})$ (Beurling case) for the convex set $Q_{R}$; see \cite[Theorem 2.14]{BG}, or \cite[Theorem 1.4.1]{Bj} and \cite[Satz 3.3]{Fie}, or \cite[Lemma 3.3]{BMT} when the weight
$\omega$ satisfies the additional assumption $\log (1+t)=o(\omega(t))$ as $t\to +\infty$.

{$(d)\Leftrightarrow(b)$}:
This is Theorem~\ref{th1A1}.

{$(d)\Rightarrow(c)$}:
It follows from Theorem~\ref{propSomega}$(a)'$, Proposition~\ref{th3A1} and Remark~\ref{add6}. 

{$(c)\Rightarrow(d)$}: It
follows from Proposition~\ref{th3A1}, Remark~\ref{add6} and Lemma~\ref{lemma312}.

{$(d)\Rightarrow(e)$}:
It is Corollary~\ref{corWig1}, since for $f\in\Sch_\omega(\R^{d})$,
\beqsn
e^{\lambda\omega(x)+\mu\omega(\xi)} \Wig f(x,\xi)\in L^{p,q}(\R^{2d})
\eeqsn
for all $\lambda,\mu\geq 0$.

{$(e)\Rightarrow(d)$}:
Follows from Lemma~\ref{lemWig10} and Corollary~\ref{corWig1}.
\end{proof}


\begin{Cor}

Given $1\leq p,q\leq+\infty$ and $R>0$,  we consider $f\in \Sch'_\omega(\R^{d})$ such that $e^{\lambda\omega(x)+\mu\omega(\xi)} \Wig f\in L^{p,q}(\R^{2d})$ for all $\lambda,\mu\geq0$. We have:
\begin{enumerate}
\item[$(a)$] $f\in \Sch_\omega(\R^{d})$ with $\supp \hat f\subseteq Q_R$ if and only if $R_{\hat f}\leq R$ and for all $\lambda,\mu>0$ there exists $C_{\lambda,\mu}>0$ such that
$$|\xi|^n_\infty\Wig f(x,\xi)<C_{\lambda,\mu} R^{n}(n+1)^{\frac d2}e^{-\lambda\omega\left(\frac{x}{n+1}\right)-\mu\omega(\xi)},\qquad n\in\N_0,\ \ \ (x,\xi)\in\R^{2d}.
$$

\item[$(b)$] $f\in \Sch_\omega(\R^{d})$ with $\supp  f\subseteq Q_R$ if and only if $R_{f}\leq R$ and for all $\lambda,\mu>0$ there exists $C_{\lambda,\mu}>0$ such that
  \beqsn
|x|_\infty|\Wig f(x,\xi)|\leq C_{\mu,\lambda}R^n(n+1)^{\frac d2}
e^{-\lambda\omega\left(\frac{\xi}{n+1}\right)-\mu\omega(x)},\qquad n\in\N_0,\ \ \ (x,\xi)\in\R^{2d}.
\eeqsn
\end{enumerate}

%
%
\end{Cor}

\begin{proof}
(a) If $f\in \Sch_\omega(\R^{d})$ and $\supp \hat{f}\subseteq Q_R$, by Theorem~\ref{th4A2}, we obtain that $f\in\PW_R^\omega(\R^{d})$. From
Proposition~\ref{propG1} we have $f\in\PWG_{R}^{\omega,\tilde f}$,
where $\tilde f(x)=f(-x)$, and hence
\beqsn
\sup_{n\in\N_0}\sup_{x,\xi\in\R^d}{(2R)}^{-n}\frac{1}{(n+1)^{d/2}}
e^{\lambda\omega\left(\frac{x}{n+1}\right)+\mu\omega(\xi)}|\xi|_{\infty}^n|V_{\tilde f}f(x,\xi)|
<+\infty.
\eeqsn

It follows from \cite[Lemma 4.3.1]{G} that
\beqsn
\lefteqn{\sup_{n\in\N_0}\sup_{x,\xi\in\R^{d}}{R}^{-n}\frac{1}{(n+1)^{d/2}}
e^{\lambda\omega\left(\frac{x}{n+1}\right)+\mu\omega(\xi)}|\xi|_{\infty}^n|\Wig f(x,\xi)|}\\
&&\le
\sup_{n\in\N_0}\sup_{x,\xi\in\R^{d}}{R}^{-n}\frac{1}{(n+1)^{d/2}}
e^{\lambda\omega\left(\frac{2x}{n+1}\right)+\mu\omega(2\xi)}\frac{|2\xi|_{\infty}^n}{2^n}
|2^{d}V_{\tilde f}f(2x,2\xi)|\\
&&=2^{d}
\sup_{n\in\N_0}\sup_{x,\xi\in\R^{d}}{(2 R)}^{-n}\frac{1}{(n+1)^{d/2}}
e^{\lambda\omega\left(\frac{x}{n+1}\right)+\mu\omega(\xi)}|\xi|_\infty^n|V_{\tilde f}f(x,\xi)|
<+\infty.
\eeqsn

Conversely, if $f\in \Sch'_\omega(\R^{d})$ with $e^{\lambda\omega(x)+\mu\omega(\xi)}\Wig f\in L^{p,q}(\R^{2d})$ and the inequality of (a) is satisfied, then $f\in \Sch_\omega(\R^{d})$ by Lemma~\ref{lemWig10} and $\supp \hat{f}\subseteq Q_R$ by
Corollary~\ref{corWig1}, since $R_{\hat f}\leq R$.

(b) It follows from (a) because
\beqsn
\sup_{n\in\N_0}\sup_{(x,\xi)\in\R^{2d}}{R}^{-n}\frac{1}{(n+1)^{\frac d2}}
e^{\lambda\omega\left(\frac{x}{n+1}\right)+\mu\omega(\xi)}|\xi|_\infty^n
|\Wig \hat{f}(x,\xi)|<+\infty
\eeqsn
is equivalent to
\beqsn
\sup_{n\in\N_0}\sup_{(x,\xi)\in\R^{2d}}{R}^{-n}\frac{1}{(n+1)^{\frac d2}}
e^{\lambda\omega\left(\frac{\xi}{n+1}\right)+\mu\omega(x)}|x|_\infty^n
|\Wig {f}(x,\xi)|<+\infty,
\eeqsn
since $\Wig \hat{f}(x,\xi)=\Wig f(-\xi,x)$ by \cite[Prop. 4.3.2]{G}.
\end{proof}


If we consider $\omega(t)=\log(1+t)$ we have that $\Sch_\omega$ is the classical
Schwartz space $\Sch$
and hence Theorem~\ref{th1A1} with $d=1$ coincides with
Theorem~1 of \cite{A1}, while Proposition~\ref{th3A1} for $d=1$ and $\lambda=0$ coincides with
Theorem~1 of \cite{B}.
We observe also that Lemma~\ref{lemma310} for $\omega(t)=\log(1+t)$ implies
\beqsn
f\in C^\infty(\R^{d}),\ (1+|x|)^\lambda f^{(\alpha)}(x)\in L^p(\R^{d})\ \forall\lambda>0,\,
\forall \alpha\in\N_0^{d},\,\mbox{for some }p\in[1,+\infty]
\Leftrightarrow f\in\Sch(\R^{d}).
\eeqsn
Moreover, Lemma~\ref{lemWig10} for $\omega(t)=\log(1+t)$ implies
\beqsn
f\in\Sch'(\R^d), &&(1+|x|)^\lambda(1+|\xi|)^\mu\Wig f(x,\xi)\in L^{p,q}(\R^d)\
\forall\mu,\lambda>0,\,\mbox{for some }p,q\in[1,+\infty]\\
\Leftrightarrow&&f\in\Sch(\R^d).
\eeqsn

The above remarks lead to the following corollary of Theorem~\ref{th4A2}
for $\omega(t)=\log(1+t)$:
\begin{Cor}
\label{corth317}
Let $1\leq p,q\leq+\infty$ and $R>0$.
Then the following conditions are equivalent:
\begin{itemize}
\item[(a)]
$f$ is an entire function in $\C^{d}$ and for all $k\in\N_0$ there exists $C_k>0$
such that
\beqsn
|f(z)|\leq C_k(1+|z|)^{-k}e^{H_{R}(\Im z)},
\quad z\in\C^{d}.
\eeqsn
\item[(b)]
$f\in \Sch(\R^{d})$ and  for all $\lambda>0$ there exists $C_\lambda>0$ such that
\beqsn
|f^{(\alpha)}(x)|\leq C_\lambda R^{\vert\alpha\vert}(|\alpha|+1)^\lambda(1+|x|)^{-\lambda}
\qquad  x\in\R^{d}, \ \alpha\in\N_0^{d}.
\eeqsn
\item[(c)]
$f\in\Sch(\R^{d})$ and
\beqsn
\lim_{n\to+\infty}\left(\max_{|\alpha|=n}\left\|f^{(\alpha)}(x)\right\|_{L^p}\right)^{1/n}=R_{\hat f}\leq R.
\eeqsn
\item[(d)]
$f\in\Sch(\R^{d})$ and $\supp\hat{f}\subseteq Q_R$.
\item[(e)]
$f\in\Sch(\R^{d})$ and
\beqsn
\lim_{n\to+\infty}\left\||\xi|^n_{\infty}\Wig f(x,\xi)\right\|_{L^{p,q}}^{1/n}=R_{\hat f}\leq R.
\eeqsn
\end{itemize}
\end{Cor}

\begin{proof}
It follows directly from Theorem~\ref{th4A2} with $\omega(t)=\log(1+t)$
and the observation that \eqref{duestelle} and
\eqref{trestelle} can be required just for $\lambda=0$ since we have
$f\in\Sch(\R^{d})$.

Note also that we can substitute
$e^{\lambda\omega\left(\frac{x}{n+1}\right)}$ with $\frac{(1+|x|)^\lambda}{(n+1)^\lambda}$
instead of $\big(1+\frac{|x|}{n+1}\big)^\lambda$ since
\beqsn
\left(\frac{1+|x|}{n+1}\right)^\lambda\leq
\left(1+\frac{|x|}{n+1}\right)^\lambda
\leq\left(1+\frac{1+|x|}{n+1}\right)^{[\lambda]+1}=
\sum_{k=0}^{[\lambda]+1}\binom{[\lambda]+1}{k}
\left(\frac{1+|x|}{n+1}\right)^k.
\eeqsn
\end{proof}

\begin{Ex}
\begin{em}
For $k\in\N_0$, let $e_k$ be the Hermite function on $\R$ defined by
\beqsn
e_k(x)=\frac{1}{(2^kk!\sqrt{\pi})^{1/2}}e^{-x^2/2}H_k(x),\qquad x\in\R,
\eeqsn
where the Hermite polynomial $H_k(x)$ of degree $k$ is given by
\beqsn
H_k(x)=(-1)^ke^{x^2}\frac{d^k}{dx^k}e^{-x^2},\qquad x\in\R.
\eeqsn

The Hermite functions $e_k\in\Sch_\omega(\R)$ (see \cite[Lemma 3.2]{L} and
\cite[Remark 4.17]{BJO1}).
Then the Wigner transform $\Wig(e_j,e_k)\in\Sch_\omega(\R^2)$ and the
Fourier-Wigner transform
\beqsn
V(e_j,e_k)(y,t):=\frac{1}{2\pi}\int_\R e_j\left(x+\frac t2\right)
e_k\left(x-\frac t2\right)e^{iyx}dx
\eeqsn
is the inverse Fourier transform of $\Wig(e_j,e_k)$ (see \cite{W}):
\beqs
\label{hermite1}
\Wig(e_j,e_k)(x,\xi)=\F(V(e_j,e_k))(x,\xi).
\eeqs

Let us denote by
\beqsn
e_{j,k}(y,t)=V(e_j,e_k)(y,t),\qquad j,k\in\N_0.
\eeqsn

By \eqref{hermite1}
\beqsn
\Wig(e_j,e_k)(x,\xi)=\hat{e}_{j,k}(x,\xi)
\eeqsn
and, by \cite[Thm. 3.4]{W}, for all $j,k\in\N_0$:
\beqsn
Le_{j,k}(y,t)=(2k+1)e_{j,k}(y,t),
\eeqsn
where $L$ is the twisted Laplacian defined by
\beqsn
L:=\left(D_y-\frac12 t\right)^2+\left(D_t+\frac12 y\right)^2.
\eeqsn

Then
\beqsn
\widehat{Le_{j,k}}=(2k+1)\hat{e}_{j,k}
\eeqsn
and, by \cite[Ex. 5.4]{BJO1}:
\beqs
\label{hermite2}
\hat{L}\hat{e}_{j,k}(x,\xi)=(2k+1)\hat{e}_{j,k}(x,\xi),
\eeqs
where
\beqsn
\hat{L}:=\left(\frac12 D_\xi+x\right)^2+\left(\frac12 D_x-\xi\right)^2.
\eeqsn

It is well-known that the Hermite functions are eigenfunctions of the
Fourier transform:
\beqsn
\hat{e}_k(\xi)=\lambda{e}_k(\xi)
\eeqsn
for some $\lambda\in\C$. Since $e_k$ does not have compact support,
we have therefore that $\hat{e}_k$ does not have compact support,
i.e. $R_{\hat{e}_k}=+\infty$. Since $e_k\in\Sch_\omega(\R)$,
by Corollary~\ref{corWig1} we have that for all $p,q\in[1,+\infty]$ and
$\mu,\lambda\geq0$:
\beqsn
\lim_{n\to+\infty}\left\|e^{\lambda\omega\left(\frac{x}{n+1}\right)+\mu\omega(\xi)}
|\xi|^n\Wig(e_k,e_k)(x,\xi)\right\|_{L^{p,q}}^{1/n}=+\infty,
\eeqsn
i.e. the eigenfunctions $\hat{e}_{k,k}=\Wig(e_k,e_k)$ of $\hat{L}$
satisfy:
\beqsn
\lim_{n\to+\infty}\left\|e^{\lambda\omega\left(\frac{x}{n+1}\right)+\mu\omega(\xi)}
|\xi|^n\hat{e}_{k,k}(x,\xi)\right\|_{L^{p,q}}^{1/n}=+\infty,
\qquad\forall\mu,\lambda\geq0.
\eeqsn

Moreover, Proposition~\ref{th3A1} and Remark~\ref{add6} imply that the Hermite functions $e_k$
satisfy
\beqsn
\lim_{n\to+\infty}\left\|e^{\lambda\omega\left(\frac{x}{n+1}\right)}
\frac{d^n}{dx^n} e_k(x)\right\|_{L^p}^{1/n}=+\infty
\eeqsn
for all $\lambda\geq0$ and $p\in[1,+\infty]$.
\end{em}
\end{Ex}

\section{Arbitrary support}
In order to characterize the support of $\hat{f}$ in terms of the growth
of some derivatives
of $f$ when $\supp\hat{f}$ is not compact, we substitute,
in the definition of $\PW_R^\omega(\R^d)$, the
derivatives $D^\alpha$ by the iterates $P(D)^n$ of a linear partial
differential operator with
constant coefficients and generalize some results of \cite{AD}.

Given a polynomial $P\in\C[\xi_1,\ldots,\xi_d]$ we denote by $P(D)$ the corresponding linear
partial differential operator with symbol $P$, where we use the standard notation $D_j:=-i\partial_j$. Following \cite{AD}, we define for an ultradistribution $T$ on $\R^d$ and a polynomial
$P\in\C[\xi_1,\ldots,\xi_d]$,
\beqs
\label{5}
R(P,T):=\sup\{|P(\xi)|:\ \xi\in\supp T\},
\eeqs
with the convention that $R(P,T)=0$ if $T\equiv0$.

\begin{Lemma}
\label{lemmapol}
Let $P\in\C[\xi_1,\ldots,\xi_d]$ be a polynomial of degree $m\geq1$.
Then, for all $k\in\N_0^d$ and $n\in\N$:
\beqs
\label{16}
D_\xi^kP(\xi)^n=\sum_{\ell=0}^{|k|}\frac{n!}{(n-\ell)!}P_{\ell,k}(\xi)P(\xi)^{n-\ell},
\eeqs
for polynomials $P_{\ell,k}(\xi)$ independent of $n$ and of degree
$\deg P_{\ell,k}\leq {\ell}(m-1)$.
\end{Lemma}

\begin{proof}
Let us prove it by induction on $|k|$. If $|k|=0$ then the statement is trivial with $P_{0,0}\equiv1$. Assume \eqref{16} to be valid for $|k|$,  and let us prove it for $|k|+1$, i.e. for a multi-index $k+\boldsymbol{e}_{j}$
for some $1\leq j\leq d$, where $\boldsymbol{e}_j$ is the vector with all entries equal to 0 except the
$j$-th entry equal to 1.

By the inductive assumption
\beqsn
D_\xi^{k+\boldsymbol{e}_j}P(\xi)^n=&&\sum_{\ell=0}^{|k|}\frac{n!}{(n-\ell)!}D_{\xi_j}[
P_{\ell,k}(\xi)P(\xi)^{n-\ell}]\\
=&&\sum_{\ell=0}^{|k|}\frac{n!}{(n-\ell)!}[D_{\xi_j}P_{\ell,k}(\xi)\cdot P(\xi)^{n-\ell}
+P_{\ell,k}(\xi)(n-\ell)\cdot D_{\xi_j}P(\xi)\cdot P(\xi)^{n-\ell-1}]\\
=&&\sum_{\ell=0}^{|k|}\frac{n!}{(n-\ell)!}D_{\xi_j}P_{\ell,k}(\xi)\cdot P(\xi)^{n-\ell}
+\sum_{\ell=0}^{|k|}\frac{n!}{(n-\ell-1)!}P_{\ell,k}(\xi)\cdot D_{\xi_j}P(\xi)\cdot P(\xi)^{n-\ell-1}
\eeqsn
with $\deg\left(P_{\ell,k}(\xi)D_{\xi_j}P(\xi)\right)\leq \ell(m-1)+(m-1)=(\ell+1)(m-1)$.

We can thus write
\beqsn
D_\xi^{k+\boldsymbol{e}_j}P(\xi)^n=\sum_{\tilde\ell=0}^{|k|+1}\frac{n!}{(n-\tilde\ell)!}
P_{\tilde\ell,k+\boldsymbol{e}_j}(\xi)P(\xi)^{n-\tilde\ell},
\eeqsn
for some polynomials $P_{\tilde\ell,k+\boldsymbol{e}_j}$ not
depending on $n$ and of degree $\deg P_{\tilde\ell,k+\boldsymbol{e}_j}\leq{\tilde{\ell}}(m-1)$.
\end{proof}

\begin{Th}
\label{th22AD}
Let $P\in\C[x_1,\ldots,x_d]$ be a polynomial of degree $m\geq1$. Let
$f\in\Sch_\omega(\R^d)$ and let $R(P,\hat{f})$ be defined as in \eqref{5}.
Then the following conditions are equivalent:
\begin{itemize}
\item[(a)]
$\forall\lambda>0$ $\exists C_\lambda>0$ such that $\forall n\in\N_0$, $x\in\R^d$
\beqs
\label{12}
|P(D)^nf(x)|\leq C_\lambda R^n e^{-\lambda\omega\left(\left|\frac{x}{n+1}\right|^{1/m}\right)};
\eeqs
\item[(b)]
$R(P,\hat{f})\leq R$.
\end{itemize}
\end{Th}

\begin{proof}
Let us first prove that $(a)\Rightarrow(b)$.
Let $\xi_0\in\R^d$ and $\varepsilon>0$ such that $|P(\xi_0)|\geq R+\varepsilon>0$.
We have to prove that $\hat{f}(\xi_0)=0$.

For every $\lambda>0$ and $n\in\N_0$ we have, from $(a)$:
\beqsn
|\hat{f}(\xi_0)|=&&\left|\frac{1}{P(\xi_0)^n}\int_{\R^d}(P(D)^nf(x))e^{-i\langle\xi_0,x\rangle}dx
\right|\\
\leq&&\frac{1}{|P(\xi_0)|^n}\int_{\R^d}C_\lambda R^n
e^{-\lambda\omega\left(\left|\frac{x}{n+1}\right|^{1/m}\right)}dx\\
=&&C_\lambda\frac{1}{|P(\xi_0)|^n}R^n(n+1)^d\int_{\R^d}e^{-\lambda\omega(|y|^{1/m})}dy\\
=&&{C^\prime_m}\left(\frac{R}{|P(\xi_0)|}\right)^n(n+1)^d,
\eeqsn
for some ${C^\prime_m}>0$, choosing $\lambda$ sufficiently large in such a way that
$e^{-\lambda\omega(\vert y\vert^{1/m})}\in L^1(\R^d)$, cf. \eqref{add2}.
Letting $n\to+\infty$ we have that $\hat{f}(\xi_0)=0$ since $|P(\xi_0)|\geq R+\varepsilon$.
Therefore $(b)$ is satisfied.

Conversely,
let us prove that $(b)\Rightarrow(a)$. By the Fourier inversion formula, for {$x\neq 0$} and $N\in\N_0$:
\beqsn
|P(D)^nf(x)|=&&\frac{1}{(2\pi)^d}\left|\int_{\R^d}P(\xi)^n\hat{f}(\xi)e^{i\langle x,\xi\rangle}
d\xi\right|\\
\leq&&\frac{1}{|x|^{2N}}\left|\int_{\R^d}P(\xi)^n\hat{f}(\xi)\Delta_\xi^Ne^{i\langle x,\xi\rangle}
d\xi\right|\\
\leq&&\frac{1}{|x|^{2N}}\int_{\R^d}\big|\Delta_\xi^N\big(P(\xi)^n\hat{f}(\xi)\big)\big|d\xi\\
\leq&&\frac{1}{|x|^{2N}}\sum_{{\vert\nu\vert=N}}\frac{N!}{{\nu!}}
\int_{|P(\xi)|\leq R}\big|D_{\xi_1}^{{2\nu_1}}\cdots D_{\xi_d}^{{2\nu_d}}\big(P(\xi)^n\hat{f}(\xi)
\big)\big|d\xi\\
\leq&&\frac{1}{|x|^{2N}}\sum_{{\vert\nu\vert=N}}\frac{N!}{{\nu!}}
\sum_{k_1=0}^{{2\nu_1}}\binom{{2\nu_1}}{k_1}\cdots\sum_{k_d=0}^{{2\nu_d}}\binom{{2\nu_d}}{k_d}\\
&&\cdot\int_{|P(\xi)|\leq R}|D_{\xi}^kP(\xi)^n|
\cdot|D_\xi^{{2\nu-k}}\hat{f}(\xi)|d\xi\\
\leq&&\frac{1}{|x|^{2N}}\sum_{{\vert\nu\vert=N}}\frac{N!}{{\nu!}}
\sum_{k_1=0}^{{2\nu_1}}\binom{{2\nu_1}}{k_1}\cdots\sum_{k_d=0}^{{2\nu_d}}\binom{{2\nu_d}}{k_d}
\sum_{\ell=0}^{|k|}\frac{n!}{(n-\ell)!}\\
&&\cdot\int_{|P(\xi)|\leq R}|P_{\ell,k}(\xi)|\cdot|P(\xi)|^{n-\ell}|D_\xi^{{2\nu-k}}\hat{f}(\xi)|d\xi,
\eeqsn
for polynomials $P_{\ell,k}(\xi)$ with $\deg P_{\ell,k}\leq {\ell}(m-1)$ independent of $n$,
by Lemma~\ref{lemmapol}.

Since $\hat{f}\in\Sch_\omega(\R^d)$ we thus have that for every
$\mu,\lambda>0$ there exists $C_{\mu,\lambda}>0$ such that
\beqs
\nonumber
|P(D)^nf(x)|\leq&&\sum_{{\vert\nu\vert=N}}\frac{N!}{{\nu!}}
\sum_{k_1=0}^{{2\nu_1}}\binom{{2\nu_1}}{k_1}\cdots\sum_{k_d=0}^{{2\nu_d}}\binom{{2\nu_d}}{k_d}
\sum_{\ell=0}^{|k|}n^\ell
\frac{1}{|x|^{2N}} R^{n-|k|}\\
\nonumber
&&\cdot\int_{|P(\xi)|\leq R}|P_{\ell,k}(\xi)|\cdot|P(\xi)|^{|k|-\ell}\cdot|D_\xi^{{2\nu-k}}\hat{f}(\xi)|d\xi\\
\nonumber
\leq&&\sum_{{\vert\nu\vert=N}}\frac{N!}{{\nu!}}
\sum_{k_1=0}^{{2\nu_1}}\binom{{2\nu_1}}{k_1}\cdots\sum_{k_d=0}^{{2\nu_d}}\binom{{2\nu_d}}{k_d}
\sum_{\ell=0}^{|k|}n^{|k|}
\frac{1}{|x|^{2N}} R^{n}\left(1+\frac 1R\right)^{2N}\\
\nonumber
&&\cdot\int_{|P(\xi)|\leq R}(1+|\xi|)^{d+1}
|P_{\ell,k}(\xi)|\cdot|P(\xi)|^{|k|-\ell}\cdot|D_\xi^{{2\nu-k}}\hat{f}(\xi)|
(1+|\xi|)^{-(d+1)}d\xi\\
\nonumber
\leq&&\sum_{{\vert\nu\vert=N}}\frac{N!}{{\nu!}}
\sum_{k_1=0}^{{2\nu_1}}\binom{{2\nu_1}}{k_1}\cdots\sum_{k_d=0}^{{2\nu_d}}\binom{{2\nu_d}}{k_d}
\sum_{\ell=0}^{|k|}n^{2N}
\frac{1}{|x|^{2N}} R^{n}\left(1+\frac 1R\right)^{2N}\\
\label{7}
&&\cdot C_{\mu,\lambda}e^{\lambda\varphi^*\left(\frac{2N-|k|}{\lambda}\right)}
e^{\mu\varphi^*\left(\frac{m|k|-\ell+d+1}{\mu}\right)}
\int_{\R^d}\frac{1}{(1+|\xi|)^{d+1}}d\xi,
\eeqs
by Theorem~\ref{thSomega}(e).

Now, since $\varphi^*$ is increasing, we have that
\beqs
\label{8}
e^{\lambda\varphi^*\left(\frac{2N-|k|}{\lambda}\right)}\leq
e^{\lambda\varphi^*\left(\frac{2mN-m|k|}{\lambda}\right)}
\eeqs
and
\beqs
\label{9}
e^{\mu\varphi^*\left(\frac{m|k|-\ell+d+1}{\mu}\right)}\leq
C_\mu e^{\frac\mu2\varphi^*\left(\frac{m|k|-\ell}{\mu/2}\right)}
e^{\frac\mu2\varphi^*\left(\frac{d+1}{\mu/2}\right)}
\leq C_{\mu,d}e^{\frac\mu2\varphi^*\left(\frac{m|k|}{\mu/2}\right)}
\eeqs
by Lemma~\ref{phistar}(ix).

Moreover, taking $\lambda=\mu/2$, we have that
\beqs
\label{10}
e^{\lambda\varphi^*\left(\frac{2mN-m|k|}{\lambda}\right)}
e^{\lambda\varphi^*\left(\frac{m|k|}{\lambda}\right)}
\leq e^{\lambda\varphi^*\left(\frac{2mN}{\lambda}\right)},
\eeqs
by Lemma~\ref{phistar}(ii).

We use now \eqref{8}, \eqref{9} and \eqref{10} in \eqref{7} to obtain that for every
$\lambda>0$ there exists $C_\lambda>0$ such that
\beqsn
|P(D)^nf(x)|\leq&&\sum_{{\vert\nu\vert=N}}\frac{N!}{{\nu!}}
\sum_{k_1=0}^{{2\nu_1}}\binom{{2\nu_1}}{k_1}\cdots\sum_{k_d=0}^{{2\nu_d}}\binom{{2\nu_d}}{k_d}
(2N+1)n^{2N}\\
&&\cdot\frac{1}{|x|^{2N}} R^{n}\left(1+\frac 1R\right)^{2N}C_\lambda
e^{\lambda\varphi^*\left(\frac{2mN}{\lambda}\right)}\\
\leq&& C_\lambda R^n\frac{1}{|x|^{2N}}d^N2^{2N}2^{2N}n^{2N}
\left(1+\frac 1R\right)^{2N}e^{\lambda\varphi^*\left(\frac{2mN}{\lambda}\right)}\\
\leq&&C_\lambda R^n\left[\frac{1}{|x|^{2N/m}}d^{2N}4^{2N}(n+1)^{\frac{2N}{m}}
\left(1+\frac 1R\right)^{2N}e^{\frac\lambda m\varphi^*\left(\frac{2N}{\lambda/m}\right)}
\right]^m.
\eeqsn

Taking the infimum over $N\in\N_0$ and applying Lemma~\ref{phistar}(vi), we have
that for all {$|x|\geq (4d)^m(n+1)(1+\frac{1}{R})^m$}

\beqs
|P(D)^nf(x)|\leq C_\lambda R^n
\bigg[e^{-\left(\frac\lambda m-\frac 2b\right)\omega\left(\frac{|x|^{1/m}}{4d(n+1)^{1/m}
\left(1+1/R\right)}\right)-\frac{2a}{b}}\bigg]^m, \label{11-1}
\eeqs
for $a\in\R, b>0$ as in condition $(\gamma)$ of Definition~\ref{defomega}.
For $|x|< (4d)^m(n+1)(1+\frac{1}{R})^m$ we have
\beqsn
\vert P(D)^nf(x)\vert &=&\frac{1}{(2\pi)^d}\left\vert \int_{\R^d}P(\xi)^n \hat{f}(\xi) e^{i\langle x,\xi\rangle}\,d\xi\right\vert \\
&\leq&\int_{\supp\hat{f}} \vert P(\xi)\vert^n \vert\hat{f}(\xi)\vert\,d\xi\leq CR^n,
\eeqsn
with $C=\Vert\hat{f}\Vert_{L^1(\R^d)}$ (observe that $C$ is finite since $\hat{f}\in\Sch_\omega(\R^d)$).
Since $\omega$ is increasing, we then have that \eqref{11-1} is satisfied for $\vert x\vert<(4d)^m (n+1)(1+\frac{1}{R})^m$
with $C_\lambda =C e^{\vert \lambda-\frac{2m}{b}\vert\omega(1)+\frac{2ma}{b}}$, and so \eqref{11-1} is satisfied for every $x\in\R^d$. From \eqref{add3}
we finally have that for every $\mu>0$ there exists $C_\mu>0$, depending on
$\mu,m,a,b,d$ and $R$, such that
\beqsn
|P(D)^nf(x)|\leq C_\mu R^n e^{-\mu\omega\left(\left|\frac{x}{n+1}\right|^{1/m}\right)},
\qquad\forall x\in\R^d,
\eeqsn
i.e. \eqref{12} is satisfied.
\end{proof}
Based in some known results of Andersen~\cite{AD}, we can deduce easily the following corollary:
\begin{Cor}\label{cor-AND}
If $P\in\C[x_1,\ldots,x_d]$ is a polynomial of degree $m\geq1$,  $f\in\Sch_\omega(\R^d)$ and $1\le p\le \infty,$ we have, for all $\lambda\ge 0$,
\begin{equation}\label{limit-with-polynomials}
\lim_{n\to+\infty}\left\|e^{\lambda\omega\left(\left|\frac{x}{n+1}\right|^{1/m}\right)} P(D)^nf(x)\right\|_{L^p}^{1/n}= R(P,\hat{f}).
\end{equation}
\end{Cor}
\begin{proof}
On one hand, from \cite[Proposition 2.4]{AD}, it is obvious that
$$
\liminf_{n\to+\infty} \left\|e^{\lambda\omega\left(\left|\frac{x}{n+1}\right|^{1/m}\right)} P(D)^nf(x)\right\|_{L^p}^{1/n}\ge R(P,\hat{f}),
$$
for all $\lambda\ge 0.$ Hence, it is sufficient to prove that
$$
\limsup_{n\to+\infty} \left\|e^{\lambda\omega\left(\left|\frac{x}{n+1}\right|^{1/m}\right)} P(D)^nf(x)\right\|_{L^p}^{1/n}\le R(P,\hat{f}),
$$
for any $\lambda\ge 0$. To see this we fix $\lambda\ge 0$ and consider $\mu>0$ big enough such that
$$
\left\| e^{-\mu\omega\left(\left|x\right|^{1/m}\right)}\right\|_{L^{p}}<+\infty.
$$
Now, we assume that $R(P,\hat{f})<+\infty.$ By Theorem~\ref{th22AD}, for every $R\ge R(P,\hat{f})$ and every $n\in\N$, we have
\begin{eqnarray*}
\lefteqn{\left\|e^{\lambda\omega\left(\left|\frac{x}{n+1}\right|^{1/m}\right)} P(D)^nf(x)\right\|_{L^p}}\\
&&\le \left\| e^{-\mu\omega\left(\left|\frac{x}{n+1}\right|^{1/m}\right)}\right\|_{L^{p}} \left\|e^{(\lambda+\mu)\omega\left(\left|\frac{x}{n+1}\right|^{1/m}\right)} P(D)^nf(x)\right\|_{L^\infty}\\
&&\le (n+1)^{d/p} C_{\lambda+\mu} \left\| e^{-\mu\omega\left(\left|x\right|^{1/m}\right)}\right\|_{L^{p}} R^{n}.
\end{eqnarray*}
We deduce that $$\limsup_{n\to+\infty} \left\|e^{\lambda\omega\left(\left|\frac{x}{n+1}\right|^{1/m}\right)} P(D)^nf(x)\right\|_{L^p}^{1/n}\le R,$$ for each $R\ge R(P,\hat{f}),$ which concludes the proof.
\end{proof}

\begin{Rem}
\label{rem1}
\begin{em}
Let us remark that Theorem~\ref{th22AD} gives an estimate, in terms of $R$,
of the upper bound of $|P(\xi)|$ for $\xi\in\supp\hat{f}$. This is interesting because
$\{\xi\in\R^d:\ |P(\xi)|\leq R\}$ can be not compact, so that we have some estimate
on the support of $\hat{f}$ for $f\in\Sch_\omega(\R^d)$, with arbitrary support
of $\hat{f}$. Our results should be compared with \cite{JJ}. See also \cite{BJ1,BJ2,BJJ}.

\end{em}
\end{Rem}

\begin{Ex}
\label{rem2}
\begin{em}
Let $P\in\C[\xi_1,\ldots,\xi_d]$ be a polynomial of degree $m\geq1$. If $P$ is hypoelliptic,
then
\beqs
\label{15}
V_R:=\{\xi\in\R^d:\ |P(\xi)|\leq R\}
\eeqs
is compact.

Indeed, if $P$ is hypoelliptic then there exist $c>0$ and $0<\sigma\leq m$ such that
\beqsn
|P(\xi)|\geq c|\xi|^\sigma,
\qquad\forall\xi\in\R^d,\ |\xi|\gg1.
\eeqsn
Therefore there exists $M>0$ such that
\beqsn
V_R\subseteq\{\xi\in\R^d:\ |\xi|\leq M\}\cup
\{\xi\in\R^d:\ c|\xi|^\sigma\leq|P(\xi)|\leq R\},
\eeqsn
and therefore
is bounded and hence compact, since its trivially closed.

On the contrary, the fact that $V_R$ is compact does not imply that $P$ is hypoelliptic.
Take, for instance,
\beqsn
P(z)=z_1^2-z_2^2+iz_2,\qquad z_1,z_2\in\C.
\eeqsn
In this case
\beqsn
V_R=\{\xi\in\R^2:\ |\xi_1^2-\xi_2^2+i\xi_2|\leq R\}
\eeqsn
is compact since $|P(\xi)|\leq R$ implies
\beqsn
&&|\Im P(\xi)|=|\xi_2|\leq R\\
&&|\Re P(\xi)|=|\xi_1^2-\xi_2^2|\leq R\quad\Rightarrow\quad
|\xi_1|\leq\sqrt{\xi_2^2+R}\leq\sqrt{R^2+R}.
\eeqsn

However, $P(\xi)$ is not hypoelliptic since the following necessary and sufficient condition
for hypoellipticity (see \cite[Prop. 2.2.1]{R}) is not satisfied:
\beqsn
\lim_{\afrac{\zeta\in V}{|\zeta|\to+\infty}}|\Im\zeta|=+\infty,
\eeqsn
for
\beqsn
V:=&&\{z\in\C^2:\ P(z)=0\}\\
=&&\{z\in\C^2:\ z_2^2-iz_2-z_1^2=0\}\\
=&&\Big\{z\in\C^2:\ z_2=\frac{i\pm\sqrt{-1+4z_1^2}}{2}\Big\},
\eeqsn
where $\pm\sqrt{-1+4z_1^2}$ denote the two complex roots of $4z_1^2-1$.

Taking, for instance,
\beqsn
\xi=\left(\xi_1,\frac{i+\sqrt{4\xi_1^2-1}}{2}\right)\in V,\qquad\mbox{for}\ \xi_1\in\R,
\eeqsn
we have that $|\xi|\to+\infty$ for $|\xi_1|\to+\infty$, but
\beqsn
|\Im\xi|=\left|\left(0,\frac12\right)\right|=\frac12.
\eeqsn
\end{em}
\end{Ex}

\vspace*{10mm}
{\bf Acknowledgments.}
The authors were partially supported by the INdAM-Gnampa Project 2017
``Equazioni a Derivate Parziali, Analisi di Gabor ed Analisi Microlocale",
by the Projects FAR 2014 and FAR 2017 (University of  Ferrara),
by the Project FFABR 2017 (MIUR). The research of the second author was partially supported by the project MTM2016-76647-P.

\end{document}